\documentclass[a4paper,11pt]{article}
\usepackage[utf8]{inputenc}
 \usepackage[T1]{fontenc}
 \usepackage[normalem]{ulem}
 \usepackage[english]{babel}
\usepackage[notref,notcite]{showkeys}
\usepackage{geometry}
\usepackage{graphicx}
\usepackage{amsmath}
\usepackage{amssymb}
\usepackage{amsthm}
\usepackage{graphicx}
\usepackage{pstricks}
\usepackage{amsfonts}
\usepackage{xr}
\usepackage{bbm} 
\usepackage{color}
\usepackage{graphicx,epsfig}
\usepackage{pdfsync,subfigure}
\usepackage{dsfont}

\newcommand{\R}{\mathbb{R}}

\renewcommand{\L}{\mathcal{L}}

\newcommand{\e}{\varepsilon}
\renewcommand{\epsilon}{\varepsilon}
\newcommand{\eps}{\varepsilon}

\newtheorem{theorem}{Theorem}

\newtheorem{proposition}[theorem]{Proposition}

\newtheorem{lemma}[theorem]{Lemma}
\newtheorem{remark}[theorem]{Remark}
\newtheorem{example}[theorem]{Example}

\title{A Hamilton-Jacobi approach for front propagation in kinetic equations}
\author{Emeric Bouin\footnote{UMR CNRS 5669 'UMPA' and INRIA project 'NUMED', \'Ecole Normale Sup\'erieure de Lyon, 
46, all\'ee d'Italie, 
F-69364 Lyon Cedex 07, 
France.}}

\begin{document}
\maketitle
\begin{abstract}
In this paper we use the theory of viscosity solutions for Hamilton-Jacobi equations to study propagation phenomena in kinetic equations. We perform the hydrodynamic limit of some kinetic models thanks to an adapted WKB ansatz. Our models describe particles moving according to a velocity-jump process, and proliferating thanks to a reaction term of monostable type. The scattering operator is supposed to satisfy a maximum principle. When the velocity space is bounded, we show, under suitable hypotheses, that the phase converges towards the viscosity solution of some constrained Hamilton-Jacobi equation which effective Hamiltonian is obtained solving a suitable eigenvalue problem in the velocity space. In the case of unbounded velocities, the non-solvability of the spectral problem can lead to different behavior. In particular, a front acceleration phenomena can occur. Nevertheless, we expect that when the spectral problem is solvable one can extend the convergence result.
\end{abstract}

\noindent{\bf AMS Class. No:} {35Q92, 45K05, 35C07}

\section{Introduction}

In this paper, we aim to study propagation phenomena in some kinetic models. The main motivation for this work comes from the study of pulse waves in bacterial colonies of \textit{Escherichia coli}. Kinetic models have been proposed to describe the run-and-tumble motion of individual bacteria at the mesoscopic scale. It has been shown recently that these kinetic models are much more accurate than their diffusion approximations, see \cite{Saragosti2} and the references therein for details. In this work, and contrary to works on chemotaxis models, we focus on propagation driven by growth effects ({\it à la} Fisher-KPP). This is one major difference between the initial motivation and this paper.
 
We consider a population of cells which is described by a probability density $f$ on $\R^+ \times \R^n \times V$, where $V$ denotes the velocity space, which is a symmetric subset of $\R^n$. We assume that the velocity of cells changes randomly following a velocity-jump process given by some operator $L$ analogous to the scattering operator in radiative transfer theory. We model the cell division with a kinetic nonlinearity of monostable type. Our kinetic model reads 
\begin{equation}\label{KinEq}
\forall (t,x,v) \in \R^+ \times \R^n \times V, \qquad \partial_t f + v \cdot \nabla_x f = L(f) + r \rho \left( M(v) - f \right), 
\end{equation}
where $r \geq 0$ stands for a \textit{growth parameter} and 
\begin{equation*}
\forall (t,x) \in \R^+ \times \R^n, \qquad \rho(t,x) := \int_V f(t,x,v) dv,
\end{equation*} 
is the macroscopic density in position $x$ at time $t$. The \textit{linear} operator $L : L^1(V) \mapsto L^1(V)$ acting only on the velocity variable describes the tumbling in the velocity space and is \textit{mass preserving}, that is 
\begin{equation*}
\forall \varphi \in L_+^1 \left( V \right), \qquad \int_V L( \varphi )(v) dv = 0.
\end{equation*}
We assume that $\text{Ker}(L) = \text{Span}(M)$, where the distribution $M \in \text{Ker}(L)$ is assumed to be nonnegative
 and satisfies 
\begin{equation*}
\int_V M(v) dv = 1, \qquad \int_{V} v M(v) dv = 0, \qquad \int_{V} v^2 M(v) dv < + \infty.
\end{equation*}
We note that $0$ and $M$ are thus stationary solutions of \eqref{KinEq}.  

A first attempt to understand the long time behavior of kinetic equations such as \eqref{KinEq} is to perform scaling limits. Due to the unbiased velocity jump process contained in our model, the diffusive limit seems particularly relevant at first glance. This issue has been particularly studied in the particular case of a BGK equation without any growth term (see \cite{Bardos} and the references therein). As a corollary, the Fisher-KPP equation can be obtained as a parabolic limit of \eqref{KinEq} when $r >0$. The long time behavior of this latter parabolic equation is now well understood since the pioneering works of Kolmogorov-Petrovskii-Piskunov \cite{Kolmogorov} and Aronson-Weinberger \cite{Aronson}. For nonincreasing inital data with sufficiently fast decay at infinity, the solution behaves asymptotically as a travelling front. It is thus natural to study propagation phenomena for kinetic equations such as \eqref{KinEq}. 

Let us emphasize that travelling wave solutions for kinetic equations raised a lot of interest recently. Caflisch and Nicolaenko construct weak shock profiles solutions of the Boltzmann equation using a micro-macro decomposition \cite{Caflisch}. Liu and Yu's main result in \cite{Liu} is the establishment of the positivity of shock profiles for the Boltzmann equation. In  \cite{Cuesta-Schmeiser}, a compactness argument as in \cite{Golse} also proves existence and positivity of big waves for a nonlinear BGK equation. The Caflisch and Nicolaenko micro-macro decomposition has been used to construct waves in a parabolic regime for a particular version of \eqref{KinEq} for the Fisher-KPP equation \cite{Cuesta}. In \cite{Bouin-Calvez-Nadin}, travelling waves have been constructed in the full kinetic regime. Golse \cite{Golse} uses compactness properties to prove existence of big waves for the kinetic Perthame-Tadmor model. 

An important technique to derive the propagating behavior in reaction-diffusion equations is to revisit the Schrödinger WKB expansion to study hyperbolic limits \cite{Freidlin, Evans-Souganidis}. Let us quickly present this approach on the standard Fisher-KPP equation, as it contains all the heuristic ideas needed to understand the present work. This equation reads
\begin{equation}\label{KPP}
\forall (t,x) \in \R^+ \times \R^n, \qquad \partial_t \rho - D \Delta_{xx} \rho = r \rho ( 1 - \rho ), 
\end{equation}
where here $x$ is the \textit{space} variable, and $r, D$ are positive parameters. In the hyperbolic limit $(t,x) \to \left( \frac{t}{\eps} , \frac{x}{\eps} \right)$, we make the so-called \textit{WKB ansatz}: 
\begin{equation}\label{WKBKPP}
\forall (t,x) \in \R^+ \times \R^n, \qquad \rho^{\eps}(t,x) = e^{- \frac{\varphi^{\eps}(t,x)}{\eps}},
\end{equation}
so that the \textit{phase} $\varphi^{\eps}$ is nonnegative and satisfies the following viscous Hamilton-Jacobi equation
\begin{equation}\label{VHJ}
\forall (t,x) \in \R^+ \times \R^n, \qquad \partial_t \varphi^{\eps} + D \vert \nabla_x \varphi^{\eps} \vert^2 + r = \eps D \Delta_x \varphi^{\eps} + r \rho^{\eps}
\end{equation}
The theory of viscosity solutions concerns the locally uniform convergence of $\varphi^{\eps}$ towards $\varphi^0$, the viscosity solution of the following so-called \textit{variational Hamilton-Jacobi equation}
\begin{equation}\label{VHJ2}
\forall (t,x) \in \R^+ \times \R^n, \qquad \min\left \lbrace \partial_t \varphi^0 + D \vert \nabla_x \varphi^0 \vert^2 + r , \varphi^0 \right\rbrace = 0.
\end{equation}
One can find rigorous justifications in \cite{Evans-Souganidis} and complements in \cite{Barles,Barles-Evans,Souganidis,Crandall}. This limit phase contains all the information we need to understand the propagating behavior. More precisely, it is possible to prove \cite{Evans,Barles-Souganidis,Fleming} that in the hyperbolic limit $\eps \to 0$, the population is contained in the nullset of the phase $\varphi^0$. The main interests of this technique is that $\varphi^{\eps}$ can be expected to be more uniformly regular than $\rho^{\eps}$, and that the full theory of Hamilton-Jacobi equations and Lagrangian dynamics can be used to understand the limit equation \eqref{VHJ2}. As an example, studying the nullset of $\varphi^0$, we recover the propagation at the minimal speed $c^* = 2 \sqrt{rD}$ for the previous Fisher-KPP equation. This fruitful WKB technique has also much been used to describe the evolution of dominant phenotypical traits in a given population (see \cite{Lorz,Bouin-Mirrahimi} and the references therein) and also to describe propagation in reaction-diffusion models of kinetic types \cite{Bouin}. 

In \cite{Bouin-Calvez}, the authors have proposed a preliminary work on a BGK equation which combines Hamilton-Jacobi equations and kinetic equations to perform the WKB approach. This latter work shows that it is necessary to stay at the kinetic level to understand the large deviation regime; One misses something while performing the WKB approach on a macroscopic approximation of the BGK equation.

In this work, we develop the results announced in \cite{Bouin-Calvez} for a wider class of linear kinetic equations. We derive rigorously the hydrodynamic limit of \eqref{KinEq} in some special situations given by the hypothesis below. Unless otherwise stated in the sequel, we suppose that $L$ takes the form:
\begin{equation*}
\forall v \in V, \qquad L(f)(v) = P(f) -  \Sigma(v) f,
\end{equation*}
where $\Sigma \in W^{1,\infty}(V)$ and $P$ is a linear operator that satisfies some structural assumptions that we specify below. The examples of such operators to keep in mind are the following 

\begin{example}

Our analysis is able to cover local and non-local situations:
\qquad
\begin{enumerate}
\item Elliptic operators with Neumann boundary conditions on $\partial V$, \textit{e.g.} the Laplacian: $L(f) = P(f) = \Delta f$, $\Sigma \equiv 0$.
\item Kernel operators: $P(f) = \int_{V} K(v,v') f(v') dv'$ and $\Sigma(v) = \int_{V} K(v',v) dv'$, where $K$ is a nonnegative kernel $\left( K \in L_+^{\infty}(V \times V) \right)$.
\end{enumerate}
\end{example}

As for the Fisher-KPP equation \eqref{KPP}, we perform the hyperbolic scaling $\left( t,x,v \right) \to \left( \frac{t}{\eps} , \frac{x}{\eps} ,v \right)$ in \eqref{KinEq}. Note that at this moment we do not rescale the velocity variable. By analogy with \eqref{WKBKPP}, our \textit{kinetic WKB ansatz} writes 
\begin{equation}\label{WKBansatz}
\forall (t,x,v) \in \R^+ \times \R^n \times V, \qquad f^{\eps}(t,x,v) = M(v) e^{-\frac{\varphi^{\eps}(t,x,v)}{\eps}}.
\end{equation}
We assume that initially 
\begin{equation*}
\forall (x,v) \in \R \times V, \qquad 0 \leq f^{\eps}(0,x,v) \leq M(v). 
\end{equation*}
As a consequence, thanks to the maximum principle of Hypothesis (H1) below, the phase $\varphi^{\eps}$ is well defined and remains nonnegative for all times.
Plugging \eqref{WKBansatz} in \eqref{KinEq}, one obtains the following equation for $\varphi^{\eps}$:
\begin{equation}\label{KinEqPhi1}
\forall (t,x,v) \in \R^+ \times \R^n \times V, \qquad \partial_t \varphi^{\eps} + v \cdot \nabla_x \varphi^{\eps} = - \frac{L\left( M(v) e^{ - \frac{\varphi^{\eps}}{\eps} }   \right)}{M(v) e^{ - \frac{\varphi^{\eps}}{\eps} }}  -  r \rho^{\eps} \left( e^{\frac{\varphi^{\eps}}{\eps}} - 1 \right).
\end{equation}
To perform the limiting equation, we would rather define the operator
\begin{equation*}
\mathcal{L}(f) = L(f) + r \left( M(v) \rho - f \right), 
\end{equation*}
and the associated decomposition
\begin{equation*}
\mathcal{P}(f) := L(f) + r M(v) \rho, \qquad \overline \Sigma := \Sigma + r.
\end{equation*}
We can now transform \eqref{KinEqPhi1} on the following form

\begin{equation}\label{KinEqPhi}
\forall (t,x,v) \in \R^+ \times \R^n \times V, \qquad  \partial_t \varphi^{\eps} + v \cdot \nabla_x \varphi^{\eps} + r = - \frac{\mathcal{L} \left( M(v) e^{ - \frac{\varphi^{\eps}}{\eps} }   \right)}{M(v) e^{ - \frac{\varphi^{\eps}}{\eps} }}  + r \rho^{\eps}.
\end{equation}
This formulation is the kinetic equivalent of what was \eqref{VHJ} for the Fisher-KPP case. We shall assume that for all $\eps > 0$, there exists a unique solution $\varphi^\eps \in \mathcal{C}_b^1\left( \R^+ \times \R^n \times V \right)$ of the Cauchy problem associated to \eqref{KinEqPhi} given some initial condition $\varphi^{\eps}(0,x,v) = \varphi_0(x) \in \mathcal{C}_b^1\left( \R^n \right)$. We stress out that if boundary conditions are needed in the velocity variable, they are implicitly contained in the definition of the operator $\L$.

We now formulate our convergence results. For this purpose, let us specify the assumptions on the different operators involved and on the velocity set $V$.

\begin{itemize}
\item[{\bf(H0)}]\label{H0} The velocity set $V \subset \R^n$ is \textbf{\textit{bounded}}.

This hypothesis is very helpful to prove Theorem \ref{HJlimit} and will be discussed and extended in Section \ref{Extensions}. 

\item[{\bf(H1)}]\label{H1} The operator $P$ satisfies a \textit{\textbf{maximum principle}}, which will be used in the following way in the sequel:

Suppose that $Q: V \mapsto V$ is nonnegative and that $u : V \mapsto V$ attains a maximum in $v^0 \in V$. Then 
\begin{equation*}\label{maxprinc}
P\left( Qu \right) (v^0) \leq P(Q)(v^0) u(v^0).
\end{equation*}
\end{itemize}

This first hypothesis is rather standard and strong but nevertheless crucial in viscosity solution procedures. It is structural and not technical. It is also helpful for space and time Lipschitz estimates, see Section \ref{Estimates}. To facilitate Lipschitz estimates in velocity, we will assume a maximum principle for the differentiated operator in velocity. Indeed, in light of the WKB ansatz \eqref{WKBansatz}, let us
assume the following 
\begin{itemize}
\item[{\bf(H2)}] There exists $\mathfrak{U}^\eps : V \mapsto V$, an operator satisfying (H1) with $\mathfrak{U}^\eps \left(1 \right) \leq 0$ and $\mathfrak{B}^\eps$ a bounded function such that,  
\begin{equation*}
\nabla_v \left( \frac{ \mathcal{P}\left( M e^{- \frac{\varphi^\eps}{\eps} } \right) }{M e^{- \frac{\varphi^\eps}{\eps} }} \right) = \mathfrak{B}^\eps - \mathfrak{U}^\eps \left( \nabla_v \varphi^\eps \right),
\end{equation*} 
\end{itemize}
This hypothesis holds for our typical examples, see the dedicated Section \ref{Estimates} below. 

\begin{example}
Let us specify Hypothesis (H2) on our typical examples. For a kernel operator of the form 
\begin{equation*}
\forall v \in V, \qquad L(f)(v) = \int_{V} K(v,v') f(v') dv' - \left( \int_V K(v,v') dv' \right) f(v), 
\end{equation*}
then
\begin{equation*} 
\mathcal{\overline P}(u) = \int_{V} \psi(\cdot,v') f(v') dv', \quad \textrm{with } \quad  \psi(v,v') = \frac{K(v,v')M(v')}{M(v)}.  
\end{equation*}
As a consequence, we have
\begin{equation*}
\mathfrak{B}^\eps = \int_V \nabla_v \psi(v,v') \left[ e^{ \frac{\varphi^\eps(v) - \varphi^\eps(v')}{\eps} } - 1\right] dv', \qquad \mathfrak{U}^\eps= - \frac{1}{\eps} \vert \nabla_v \varphi^\eps \vert \int_V \psi(v,v')  \left[ e^{ \frac{\varphi^\eps(v) - \varphi^\eps(v')}{\eps} }\right] dv'.
\end{equation*}
Hypothesis (H2) is satisfied after Proposition \ref{estimate} (i), (ii), (iii). 
%
Now we come to elliptic operators. As an example, let us consider
\begin{equation*}
\forall v \in V, \qquad L(f)(v) = \nabla_v \left( D \nabla_v f \right),
\end{equation*}
with Neumann boundary conditions on $\partial V$. The diffusivity matrix $D$ is here assumed to be positive definite. Then the hypothesis (H1) and (H2) are well satisfied, with 
\begin{equation*}
\mathfrak{B^\e} = 0, 
\quad \mathfrak{U^\e}\left( \cdot  \right) = \frac{1}{\e} \nabla_v\left( D \nabla_v \left( \cdot \right) \right) - \frac{1}{\e^2}
\left\langle \cdot 
\vert \left( D + D^\top \right) \nabla_v \left( \cdot \right)  
\right\rangle.
\end{equation*}


\end{example}

We finally need to state a structural hypothesis on $\mathcal{P}$ in order to characterize the behavior with respect to $v$ in the limit. Roughly speaking, we need coercivity. 
\begin{itemize}
\item[{\bf(H3)}] There exists a linear operator $\mathfrak{\overline U}^\eps$ which satisfies the maximum principle of Hypothesis (H1), a continuous and nonnegative Hamiltonian $\mathfrak{N} : \R \times \R \mapsto \R^+$ such that every viscosity solution of $\mathfrak{N}\left( u , \nabla_v u \right) = 0$ is constant, and $\alpha, \beta > 0$, such that the following inequality holds true
\begin{equation*}
\forall v \in V, \qquad \mathfrak{N}\left( \varphi^\eps , \nabla_v \varphi^\eps \right) - \e^\alpha \mathfrak{\overline U^\e}\left( \varphi^\eps \right) \leq \eps^\beta \left \vert \frac{ \mathcal{P} \left( M e^{- \frac{\varphi^\eps}{\eps} } \right) }{ M e^{- \frac{\varphi^\eps}{\eps} }} \right \vert.  
\end{equation*}
\end{itemize}
\begin{example}
For a kernel operator, one has
\begin{equation*}
\mathfrak{N}(u,\nabla_v u) = \int_V \psi(v,v') \left\vert u(v) - u(v') \right\vert_+   dv', \qquad \mathfrak{\overline{U}^\eps} \equiv 0.
\end{equation*}
For the Laplacian equation, one has 
$\mathfrak{N}(u,\nabla_v u) = \vert \nabla_v u \vert^2 $ and $\mathfrak{\overline{U}^\eps} \equiv \Delta$. 
\end{example}

Let us now state our kinetic convergence result in the Theorem \ref{HJlimit} below. The main difficulty in the kinetic framework is to understand what to do with the velocity variable in the limit $\eps \to 0$. Roughly speaking, we will show that up to extraction, $\varphi^{\eps}$ converges towards a viscosity solution of an Hamilton-Jacobi equation, whose effective Hamiltonian is obtained through an eigenvalue problem in the velocity variable that we write in (H4) below. In fact, the limiting phase $\varphi^0$ will be independent from the velocity variable, but the kinetic nature of the $\eps$-problem is contained in this following spectral problem. We notice finally that, the roles of the velocity variable $v$ and the spectral problem in (H4) below are respectively similar to  the ones of the fast variable and the cell problem in homogenization theory.\\

\begin{itemize}
\item[{\bf(H4)}]{\textbf{\textit{Spectral problem.}}}\label{EVpb}
For all $p \in \R^n$, there exists a unique $\mathcal{H}(p)$ such that there exists a positive normalized eigenvector $Q_p \in L^1(V)$ such that 
\begin{equation}\label{eigenpb}
\forall v \in V, \qquad \mathcal{L}(Q_p)(v) + \left( v \cdot p \right) Q_p(v) = \mathcal{H}(p) Q_p(v).
\end{equation}
Moreover, $\mathcal{H}$ and $Q_p$ are smooth functions of $p$.
\end{itemize}

Section \ref{EvPb} is devoted to giving relevant conditions on the operator $\mathcal{L}$ which ensure that \eqref{eigenpb} has a solution. We also provide there some classical examples. We are now ready to state the main result:
 
\begin{theorem}{\textbf{\text{Hamilton-Jacobi limit.}}}\label{HJlimit}

Let $V$ be a symmetric subset of $\, \R^n$ satisfying (H0), $M \in L^1(V)$ be nonnegative and symmetric and $r \geq 0$. Suppose that the initial data is well-prepared,  
\begin{equation*}
\forall (x,v) \in \R^n \times V, \qquad \varphi^{\eps}(0,x,v) = \varphi_0(x),
\end{equation*}
and that the Hypotheses (H1), (H2), (H3) and (H4) are satisfied. Then, $\left( \varphi^{\eps} \right)_\eps$ converges locally uniformly towards $\varphi^0$, where $\varphi^0$ does not depend on $v$. Moreover $\varphi^0$ is the unique viscosity solution of one of the following Hamilton-Jacobi equations:

\begin{itemize}
\item[(i)] If $r=0$, then $\varphi^0$ solves the standard Hamilton-Jacobi problem
\begin{equation}\label{standHJ}
\begin{cases}
\partial_t \varphi^0 + \mathcal{H} \left(\nabla_x \varphi^0 \right) = 0, \qquad  \forall (t,x) \in \R_+^* \times \R^n, \medskip\\
\varphi^{0}(0,x) = \varphi_0(x), \qquad x \in \R^n.
\end{cases}
\end{equation}

\item[(ii)] If $r > 0$, then the limiting equation is the following constrained Hamilton-Jacobi equation
\begin{equation}\label{varHJ}
\begin{cases}
\min\left \lbrace \partial_t \varphi^0 + \mathcal{H} \left(\nabla_x \varphi^0 \right) + r , \varphi^0 \right\rbrace = 0, \qquad  \forall (t,x) \in \R_+^* \times \R^n, \medskip \\
\varphi^{0}(0,x) = \varphi_0(x), \qquad x \in \R^n.
\end{cases}
\end{equation}
\end{itemize}
where in both cases $\mathcal{H}(p)$ is an Hamiltonian given by (H4).
\end{theorem}

We point out that the assumption concerning the non-dependency on $v$ of the initial data $\varphi^{\eps}(t = 0,\cdot)$ in Theorem \ref{HJlimit} is to avoid a boundary layer in $t=0$ when $\eps \to 0$. The result can be easily extended to the case of an initial condition with small velocity perturbations, that is $\lim_{\eps \to 0} \varphi^\eps(0,x,v) = \varphi_0(x)$ uniformly in $(x,v) \in \R \times V$.

Our paper is organized as follows. The following Section \ref{Estimates} proves $W^{1,\infty}$ type estimates on $\varphi^{\eps}$ after assuming Hypothesis (H1) and (H2). In Section \ref{HJProof}, we provide the proof of Theorem \ref{HJlimit}. We dedicate Section \ref{EvPb} to solving the eigenvalue problem of Hypothesis (H4) which gives the Hamiltonian $\mathcal{H}$ in some particular situations. We conclude this first part of results with a Section \ref{Asymptotics}, giving refined asymptotics on $\varphi^\eps$, and recalling some elements to study the speed of propagation of the fronts when the constrained Hamilton-Jacobi equation \eqref{varHJ} is derived, following \cite{Evans-Souganidis,Freidlin,Fedotov}. The last Section \ref{Extensions} is devoted to discussing the results when the velocity set is unbounded. We put forward the fact that when the spectral problem of Hypothesis (H4) is not solvable, a front acceleration can occur. Finally, we show two cases for which Hypothesis (H4) holds and where we expect the convergence result to be also true in the whole space despite additional difficulties.

\section{The phase $\varphi^\eps$ is uniformly Lipschitz.}{\label{Estimates}}

In this Section, we derive some \textit{a priori} estimates on $\varphi^{\eps}$ mainly thanks to the maximum principle contained in Hypothesis (H1) and (H2).

\begin{proposition}\label{estimate}
Let $r\geq 0$ and $\varphi^{\eps} \in \mathcal{C}_b^1 \left( \R^+ \times \R \times V \right)$ a solution of equation \eqref{KinEqPhi}. Suppose that (H0) and the structural assumptions on $\L$, (H1) and (H2), hold. Then the phase $\varphi^\eps$ is uniformly locally Lipschitz. Precisely the following \textit{a priori} bounds hold:

$\exists \, C > 0, \forall t \in \R^+$,
\begin{equation*}
\begin{array}{lclc}
(i) & 0 \leq \varphi^{\eps}(t, \cdot) \leq \Vert \varphi^0 \Vert_{\infty}, & (ii) &\Vert \nabla_x \varphi^{\eps} ( t , \cdot )\Vert_{\infty} \leq \Vert \nabla_x \varphi^{0} \Vert_{\infty}, \medskip\\
(iii) & \Vert \partial_t \varphi^{\eps} (t, \cdot ) \Vert_\infty \leq V_{\text{max}} \Vert \nabla_x \varphi_{0} \Vert_\infty, & (iv) &  \quad \Vert \nabla_v  \varphi^{\eps} (t,\cdot)\Vert_{\infty} \leq C t. \\
\end{array}
\end{equation*}
\end{proposition}

\begin{proof}[\bf Proof of Proposition \ref{estimate}]

Let us first prove (i). We define $\psi_{\delta}^{\eps} (t,x,v) = \varphi^{\eps}(t,x,v) - \delta t - \delta^4 \vert x \vert^2$. As $V$ is bounded and $\psi_{\delta}^{\eps} $ is coercive in the space-time variable, for any $\delta > 0$, $\psi_{\delta}^{\eps}$ attains a maximum at point $(t_{\delta},x_{\delta},v_{\delta})$. Suppose that $ t_{\delta} > 0$. Then, we have 
\begin{equation*} 
\partial_t \varphi^{\eps} (t_{\delta},x_{\delta},v_{\delta}) \geq \delta, \qquad \nabla_{x} \varphi^{\eps} (t_{\delta},x_{\delta},v_{\delta}) = 2\delta^4 x_{\delta}.
\end{equation*}
Moreover, thanks to the maximum principle of the operator $P$, we get:
\begin{equation*}
\begin{array}{lcl}
\mathcal{L} \left( M e^{ - \frac{\varphi^{\eps} (t_{\delta},x_{\delta},\cdot) }{\eps} } \right)( v_{\delta} ) & \geq & 0.
\end{array}
\end{equation*}
As a consequence, we have at the maximum point $(t_{\delta},x_{\delta},v_{\delta})$:
\begin{equation}\label{ppmax}
0 \geq - \frac{\mathcal{L}(M e^{ - \frac{\varphi^{\eps} (t_{\delta},x_{\delta},\cdot) }{\eps} })(v_{\delta})}{M(v_{\delta}) e^{ - \frac{\varphi^{\eps}}{\eps} }} + r \left( \rho^{\eps} - 1 \right)
 \geq \delta + 2 \delta^4 v_{\delta}  \cdot x_{\delta}  \geq \delta - 2 \delta^4 V_{max} \vert x_{\delta}\vert.
\end{equation}
From what we deduce
\begin{equation}\label{cont2}
\vert x_{\delta}\vert \geq \frac{1}{2 \delta^3 V_{max}}.
\end{equation}
Moreover, the maximal property of $(t_{\delta},x_{\delta},v_{\delta})$  also implies 
\begin{equation*}
\Vert  \varphi^{\eps} \Vert_{\infty} - \delta^4 \vert x_{\delta} \vert^2  \geq \varphi^{\eps}(t_{\delta},x_{\delta},v_{\delta}) - \delta t_{\delta} - \delta^4 \vert x_{\delta} \vert^2  \geq \varphi^{\eps}(0,0,0) \geq 0 \,,
\end{equation*} 
and this gives 
\begin{equation}\label{cont1}
\vert x_{\delta}\vert \leq \frac{\Vert  \varphi^{\eps} \Vert_{\infty}^{\frac12}}{ \delta^2}.
\end{equation}
Gathering \eqref{cont1} and \eqref{cont2}, we obtain a contradiction since both cannot hold  
for sufficiently small $\delta>0$. As a consequence  $t_{\delta} = 0$, and we have,
\begin{equation*}
\forall (t,x,v) \in [0,T] \times \R^n \times V, \qquad \varphi^{\eps}(t,x,v) \leq \varphi^0(x_{\delta},v_{\delta}) + \delta t + \delta^4 \vert x \vert^2 \leq \Vert \varphi_0 \Vert_{\infty} + \delta t + \delta^4 \vert x \vert^2.
\end{equation*}
Passing to the limit $\delta \to 0$, we obtain the claim (i).

We now come to the proof of (ii). We also use maximum principle arguments, which are possible without any supplementary hypothesis on the structure of the operator $\mathcal{L}$ since this latter operator just acts on the velocity variable. Differentiating equation \eqref{KinEqPhi} with respect to the space variable, we obtain  
\begin{equation}\label{estGx1}
\begin{array}{lcl}
\left( \partial_t + v \cdot \nabla_x \right) \left( \nabla_x \varphi^{\eps} \right) & = & \frac{1}{\eps} \left( \frac{  \mathcal{L} \left( M e^{ - \frac{\varphi^{\eps}(t,x,\cdot)}{\eps} } \nabla_x \varphi^{\eps}(t,x,\cdot) \right) - \mathcal{L} \left( M e^{ - \frac{\varphi^{\eps}(t,x,\cdot)}{\eps} } \right) \nabla_x \varphi^{\eps}  }{M e^{ - \frac{\varphi^{\eps}}{\eps} } } \right) + r \nabla_ x \rho^{\eps}. \\&&
\end{array}
\end{equation}
In order to control $\nabla_x \rho^\eps$ when $r > 0$, we need to extract from $\L$ the BGK part we put in (and thus come back to the main operator $L$). Doing so and now testing \eqref{estGx1} on  
$\text{sgn}\left( \partial_{x_i} \varphi^{\eps} \right) e_i$, we obtain:
\begin{multline}\label{estGx2}
\left( \partial_t + v \cdot \nabla_x \right) \left( \vert \partial_{x_i} \varphi^{\eps} \vert \right) + \frac{r }{\eps} \rho^{\eps} \vert \partial_{x_i} \varphi^{\eps} \vert \leq \\ \frac{1}{\eps} \left( \frac{  L \left( M e^{ - \frac{\varphi^{\eps}(t,x,\cdot)}{\eps} } \partial_{x_i} \varphi^{\eps} (t,x,\cdot) \right) \text{sgn}\left( \partial_{x_i} \varphi^{\eps} \right)  - L \left( M e^{ - \frac{\varphi^{\eps}(t,x,\cdot)}{\eps} } \right)  \vert \partial_{x_i} \varphi^{\eps} \vert   }{M e^{ - \frac{\varphi^{\eps}}{\eps} } } \right)\\ - \frac{r}{\eps}\left( e^{ \frac{\varphi^{\eps}}{\eps} } - 1 \right) \int_{V} M(v') e^{ - \frac{\varphi^{\eps} (v') }{\eps}      } \left( \vert \partial_{x_i} \varphi^{\eps} (v) \vert - \text{sgn}\left( \partial_{x_i} \varphi^{\eps}(v) \right) \partial_{x_i} \varphi^{\eps} (v') \right) dv',
\end{multline}
%
since
\begin{equation*}
\text{sgn}\left( \partial_{x_i} \varphi^{\eps} \right) \nabla_x \rho^{\eps} \cdot e_i + \frac{\rho^{\eps}}{\eps} \vert \partial_{x_i} \varphi^{\eps} \vert = \frac{1}{\eps} \int_{V} M(v') e^{ - \frac{\varphi^{\eps} (v') }{\eps}      } \left( \vert \partial_{x_i} \varphi^{\eps} (v) \vert - \text{sgn}\left( \partial_{x_i} \varphi^{\eps}(v) \right) \partial_{x_i} \varphi^{\eps} (v') \right) dv'.
\end{equation*}

As for the uniform bound on $\varphi^\eps$, we conclude by performing a $\delta -$correction argument. Define, for a positive $\delta$, the auxiliary function $\psi_{\delta , i }^{\eps} = \vert \partial_{x_i} \varphi^{\eps} \vert - \delta t - \delta^4 \vert x \vert^2$. It attains a maximum in $(t,x,v)_{\delta}$. Let us now consider the two r.h.s of \eqref{estGx2} separately. First, in $(t,x,v)_{\delta}$, one has thanks to the maximum principle \eqref{maxprinc} and the independency of $L$ from the $(t,x)$ variables
\begin{equation*}
L \left( M e^{ - \frac{\varphi^{\eps}(t,x,\cdot)}{\eps} } \partial_{x_i} \varphi^{\eps} (t,x,\cdot) \right) \text{sgn}\left( \partial_{x_i} \varphi^{\eps} \right)  - L \left( M e^{ - \frac{\varphi^{\eps}(t,x,\cdot)}{\eps} } \right)  \vert \partial_{x_i} \varphi^{\eps} \vert  \leq 0.
\end{equation*}
Then, to prove that the second part of the r.h.s is nonpositive, we write
\begin{multline*}
\forall v' \in V, \qquad \text{sgn}\left( \partial_{x_i} \varphi^{\eps}(v_\delta) \right) \partial_{x_i} \varphi^{\eps} (v') - \delta t_\delta - \delta^4 \vert x_\delta \vert^2 \\ \leq  \vert \partial_{x_i} \varphi^{\eps} (v') \vert - \delta t_\delta - \delta^4 \vert x_\delta \vert^2 \leq \vert \partial_{x_i} \varphi^{\eps} (v_\delta) \vert - \delta t_\delta - \delta^4 \vert x_\delta \vert^2,
\end{multline*}
which gives the property
\begin{equation*}
\forall v' \in V, \qquad \vert \partial_{x_i} \varphi^{\eps} (v_\delta) \vert - \text{sgn}\left( \partial_{x_i} \varphi^{\eps}(v_\delta) \right) \partial_{x_i} \varphi^{\eps} (v') \geq 0.
\end{equation*}
Combining these two inequalities give, at the point of maximum:
\begin{equation*}
\delta + 2 \delta^4 v_{\delta}  \cdot x_{\delta} + \frac{r\rho^{\eps}}{\eps} \vert \partial_{x_i} \varphi^{\eps} \vert \leq 0.
\end{equation*}
The conclusion is similar to the uniform bound of $\varphi^\eps$: The maximum cannot be attained elsewhere than in $t_\delta = 0$, and the estimate is proved.

With exactly the same method, we get that necessarily $ \Vert \partial_t \varphi^{\eps} \Vert_{\infty} \leq \vert \partial_t \varphi^{\eps} ( 0 )\vert $. But, passing to the limit $ t \to 0 $ in \eqref{KinEqPhi1}, and since $\varphi^0$ does not depend on $v$, one gets $ \vert \partial_t \varphi^{\eps} ( 0 )\vert  \leq V_{max} \Vert \nabla_x \varphi_0 \Vert_{\infty}$. This gives (iii).

We finally come to the proof of the bound on the velocity gradient. This proof clearly requires a supplementary assumption on the operator $\mathcal{L} $ to be able to write an useful equation on $\vert \nabla_v \varphi^\eps\vert$. We have made the choice of a maximum principle for the derivative operator. Again, differentiating \eqref{KinEqPhi} with respect to $v$ and using Hypothesis (H2) yield
\begin{equation*}
\begin{array}{lcl}
\left( \partial_t + v \cdot \nabla_x \right) \left( \nabla_v \varphi^{\eps} \right) + \nabla_x \varphi^\eps & = & - \nabla_v \left( \frac{\mathcal{\overline P}\left( e^{ - \frac{\varphi^{\eps}}{\eps} }   \right)}{e^{ - \frac{\varphi^{\eps}}{\eps} }}  \right) + \nabla_v \Sigma \\
& = & \nabla_v \Sigma - \mathfrak{B^\e} + \mathfrak{U^\e} \left( \nabla_v \varphi^\eps \right)
\end{array}
\end{equation*}
We now test against $\frac{\nabla_v \varphi^\eps}{\vert \nabla_v \varphi^\eps \vert}$:
\begin{equation*}
\left( \partial_t + v \cdot \nabla_x \right) \left( \vert \nabla_v \varphi^{\eps} \vert \right)  \leq  \Vert \nabla_x \varphi^0 \Vert_{\infty} + \Vert \nabla_v \Sigma\Vert_{\infty} + \Vert \mathfrak{B^\e} \Vert_{\infty} +  \mathfrak{U^\e} \left( \nabla_v \varphi^\eps \right) \cdot \frac{\nabla_v \varphi^\eps}{\vert \nabla_v \varphi^\eps \vert}.
\end{equation*}
Thanks to the maximum principle (H1), we deduce that there exists $C$ such that $\Vert \nabla_v \varphi^\eps \Vert_\infty \leq \Vert \nabla_v \varphi^\eps (t=0) \Vert_\infty + C t = Ct$, as we supposed that the initial data does not depend on $v$, and this proves (iv).

\end{proof}

\section{Hamilton - Jacobi dynamics - Proof of Theorem \ref{HJlimit}.}\label{HJProof}

In this Section, we present the proof of our main result, Theorem \ref{HJlimit}. We divide the proof into two parts. We first show that the structural assumptions on the operator $\mathcal{L}$ make $\varphi^\eps$ converge locally uniformly up to a subsequence towards a function independent of the velocity variable, which is the first point of Theorem \ref{HJlimit}. Then, we perform our kinetic Hamilton-Jacobi procedure to identify the limit as a solution of one of the Hamilton-Jacobi equations \eqref{standHJ} or \eqref{varHJ}.

\subsection{Convergence of $\varphi^\eps$.}

For the convenience of the reader, we enlighten the convergence property in the following 
\begin{proposition}\label{conv}
Suppose that (H0), (H1), (H2) and (H3) hold. Then, up to a subsequence, the phase $\varphi^\eps$ converges locally uniformly in $\R^+ \times \R^n \times V$ towards $\varphi^0$, which does not depend on $v$.
\end{proposition}

\begin{proof}[{\bf Proof of Proposition \ref{conv}}]
Given the assumptions (H0), (H1) and (H2), we deduce from Proposition \ref{estimate}, Ascoli's theorem that in all compact subsets of $\R^+ \times \R^n \times V$, we can extract from $\varphi^\eps$ a converging subsequence. The limit $\varphi^0$ is uniquely defined on the whole space after increasing extraction on compacts. 
The uniform bounds of Proposition \ref{estimate} also give that $ \left\vert \frac{\mathcal{{L}} \left( M(v) e^{ - \frac{\varphi^{\eps}}{\eps} }   \right)}{M(v) e^{ - \frac{\varphi^{\eps}}{\eps} }} \right\vert$ is uniformly bounded by a constant $C$. We thus deduce from (H3) that for all $(t,x) \in \R^+ \times \R^n$,
\begin{equation*}
\forall v \in V, \qquad \mathfrak{N}\left(\varphi^\eps, \nabla_v \varphi^\eps \right) \leq \eps^{\alpha} \mathfrak{\overline{U}^\e}\left(\varphi^\eps\right) + \eps^\beta C. 
\end{equation*}
Hence, since $\mathfrak{\overline{U}^\e}$ satisfies the maximum principle (H1), one obtains when $\eps \to 0$ that $u := \varphi^0(t,x,\cdot)$ is a viscosity solution of 
\begin{equation*}
\mathfrak{N}\left( u , \nabla_v u \right) = 0.
\end{equation*}
Recall that $\mathfrak{N}$ is positive. Thus $u$ is constant thanks to Hypothesis (H3).

\end{proof}

\subsection{Identification of the limit.}

In this Subsection, we present the viscosity procedure which identifies the viscosity limit of $\varphi^\eps$. We will follow the same steps as in the seminal paper of Evans and Souganidis \cite{Evans-Souganidis}. In addition with a relevant use of corrected tests functions, see \cite{Evans}. Indeed, the resolution of the spectral problem of Hypothesis (H4) is of main importance to define a corrector in the viscosity procedure, see \eqref{correction1} and \eqref{eqeta}.

Since we already know that $\varphi^{\eps} \geq 0$, the remaining properties to be proven to get the result of Theorem \ref{HJlimit} are gathered in the two following steps: 

\bigskip

{\bf \# Step 1: Viscosity supersolution.}

\qquad 

The statement of the supersolution property does not depend explicitely on the growth part.
 
\begin{lemma}\label{super}
Assume $r \geq 0$. Then $\varphi^0$ satisfies 
\begin{equation}\label{HJeq}
\forall (t,x) \in \R^{+*} \times \R^n, \qquad \partial_t \varphi^0 + \mathcal{H} \left( \nabla_x \varphi^0 \right) + r \geq 0.
\end{equation}
in the viscosity sense.
\end{lemma}

\begin{proof}[\bf Proof of Lemma \ref{super}]
Let $\psi^0 \in \mathcal{C}^2 \left( \R^+ \times \R^n \right)$ be a test function such that $ \varphi^0 - \psi^0 $ has a strict local minimum a $(t^0,x^0)$ with $t^0 > 0$. We want to show that
\begin{equation*}
\qquad \partial_t \psi^0 (t^0,x^0) + \mathcal{H} \left( \nabla_x \psi^0 (t^0,x^0) \right) + r\geq 0.
\end{equation*}
We define the corrected test functions \cite{Evans,Crandall-Evans} by 
\begin{equation}\label{correction1}
\forall (t,x,v) \in \R^{+*} \times \R^n \times V, \qquad \psi^{\eps}(t,x,v) := \psi^{0}(t,x) + \eps \eta(t,x,v), 
\end{equation}
with a correcting term $\eta$ that comes after Hypothesis (H4). Indeed, we set: 
\begin{equation}\label{eqeta}
\forall (t,x,v) \in \R^+ \times \R^n \times V, \qquad \eta(t,x,v) = - \ln \left( \frac{ Q_{\left[\nabla_x \psi^{0}(t,x)\right]}(v) }{M(v)} \right).
\end{equation}


The definition of the correcting function gives that $\varphi^{\eps} - \psi^{\eps}$ converges locally uniformly towards $\varphi^0 - \psi^0$. As a consequence, there exists a sequence $(t^{\eps},x^{\eps}) \in \R^{+*} \times \R^n$ of strict local minima in $(t,x)$ which converges towards $(t^0,x^0)$ and a sequence $v^\eps \in V$ such that $(t^{\eps},x^{\eps},v^\eps)$ minimizes $\varphi^{\eps} - \psi^{\eps}$. At the point $(t^\eps,x^{\eps},v^{\eps})$, using the spectral problem of (H4) with $p^{\eps}= \nabla_x \psi^{\eps} (t^\eps,x^{\eps},v^{\eps})$, one obtains:
\begin{equation*}
\partial_t \psi^{\eps}  + \mathcal{H} \left( p^{\eps}  \right) + r  \; = \;  \partial_t \psi^{\eps}  + v^{\eps} \cdot p^{\eps} + \frac{\mathcal{L}\left(Q_{p^{\eps}} \right)(v^{\eps})}{Q_{p^{\eps}} }+ r.
\end{equation*} 
We notice that at the point $(t^{\eps},x^{\eps},v^{\eps})$, the following holds 
\begin{equation*}
\partial_t \varphi^{\eps} = \partial_t \psi^{\eps}, \qquad \nabla_x \varphi^{\eps} = \nabla_x \psi^{\eps} = p^{\eps}.
\end{equation*}
Thus,
\begin{equation*}
\begin{array}{lcl}
\partial_t \psi^{\eps}  + \mathcal{H}\left( p^{\eps} \right) + r  & = & \partial_t \varphi^{\eps} + v^{\eps} \cdot \nabla_x \varphi^{\eps}  + r + \frac{\mathcal{L}\left(Q_{p^{\eps} }  \right)(v^{\eps})}{Q_{p^{\eps}} },\\
\\
& = & \frac{\mathcal{L}\left(Q_{p^{\eps} } \right)(v^{\eps})}{Q_{p^{\eps}}   } - \frac{\mathcal{L} \left( Me^{ - \frac{\varphi^{\eps}(t^\eps,x^{\eps},\cdot)}{\eps} }   \right)(v^{\eps})}{M e^{ - \frac{\varphi^{\eps}}{\eps} }}  + r \rho^{\eps}(t^\eps,x^{\eps}),\\
\end{array}
\end{equation*}
recalling \eqref{KinEqPhi}. Recall that potential boundary conditions are included in the formulation of the operators. Simplifying the latter and using $\rho^{\eps} \geq 0$, we obtain at the point $(t^{\eps},x^{\eps},v^{\eps})$:
\begin{equation*}
\begin{array}{lcl}
\partial_t \psi^{\eps}  + \mathcal{H}\left( \nabla_x \psi^{\eps}  \right) + r   & \geq & \frac{\mathcal{P}\left(Q_{p^{\eps}}  \right)(v^{\eps})}{Q_{p^{\eps} }   } - \frac{ \mathcal{P} \left( M e^{ - \frac{\varphi^{\eps}(t^\eps,x^{\eps},\cdot)}{\eps} }   \right)(v^{\eps})}{     M e^{ - \frac{\varphi^{\eps}}{\eps} } }.\\
\end{array}
\end{equation*}

But, from the minimal character of $(t^{\eps},x^{\eps},v^{\eps})$ and the maximum principle satisfied by $\mathcal{P}$ we deduce that the following holds at the point $(t^{\eps},x^{\eps},v^{\eps})$:

\begin{equation*}
\begin{array}{lcl}
- \frac{ \mathcal{P} \left( M(\cdot) e^{ - \frac{\varphi^{\eps}(t^\eps,x^{\eps},\cdot)}{\eps} }  \right)(v^{\eps})}{     M e^{ - \frac{\varphi^{\eps}}{\eps} }  }
& = &  - \frac{ \mathcal{P} \left( M e^{  - \frac{  \varphi^{\eps} - \psi^{\eps}   }{\eps} (t^\eps,x^{\eps}, \cdot)}   e^{- \frac{\psi^{0}(t^\eps,x^{\eps})}{\eps} } e^{- \eta(t^\eps,x^{\eps},\cdot) }  \right)(v^{\eps})}{ M e^{- \frac{\varphi^{\eps}  - \psi^{\eps} }{\eps}} e^{-  \frac{\psi^{0}(t^\eps,x^{\eps})}{\eps} } e^{- \eta } },\medskip \\
& = &  - \frac{ \mathcal{P} \left( M e^{- \eta(t^\eps,x^{\eps},\cdot) }  e^{  - \frac{  \varphi^{\eps} - \psi^{\eps}   }{\eps} (t^\eps,x^{\eps}, \cdot)}    \right)(v^{\eps})}{ M e^{- \eta } e^{- \frac{\varphi^{\eps}  - \psi^{\eps} }{\eps}}  },\medskip \\
& = &  - \frac{ e^{- \frac{\varphi^{\eps}  - \psi^{\eps} }{\eps} (t^\e, x^\e,v^\e)}  \mathcal{P} \left( M e^{- \eta(t^\eps,x^{\eps},\cdot) }     \right)(v^{\eps})}{ M e^{- \eta } e^{- \frac{\varphi^{\eps}  - \psi^{\eps} }{\eps}}  },\medskip \\
& \geq & - \frac{ \mathcal{P} \left( M   e^{-  \eta(t^\eps,x^{\eps},\cdot) }  \right)(v^{\eps})}{ M  e^{-  \eta }  } .
\end{array}
\end{equation*}
One deduces, at the point $(t^\eps,x^{\eps},v^{\eps})$: 
\begin{equation*}
\partial_t \psi^{\eps}  + \mathcal{H} \left( \nabla_x \psi^{\eps}  \right) + r \; \geq \; \frac{\mathcal{P}\left(Q_{\nabla_x \psi^{\eps} (t^\eps,x^{\eps},v^{\eps}) } \right)(v^{\eps})}{Q_{\left[\nabla_x \psi^{\eps} (t^\eps,x^{\eps},v^{\eps}) \right]}   } - \frac{ \mathcal{P}\left( Me^{-  \eta(t^\eps,x^{\eps},\cdot) }  \right)(v^{\eps})}{ M e^{-  \eta }  } 
\end{equation*}
Here comes the specification of the corrector $\eta$. We obtain, at the point $(t^\eps,x^{\eps},v^{\eps})$:
\begin{equation*}
\partial_t \psi^{\eps}  + \mathcal{H} \left( \nabla_x \psi^{\eps}  \right) + r  \geq \frac{\mathcal{P}\left(Q_{\nabla_x \psi^{\eps} (t^\eps,x^{\eps},v^{\eps}) } \right)(v^{\eps})}{Q_{\left[ \nabla_x \psi^{\eps} (t^\eps,x^{\eps},v^{\eps}) \right]}  } - \frac{\mathcal{P}\left(Q_{\nabla_x \psi^{0} (t^\eps,x^{\eps}) }  \right)(v^{\eps})}{Q_{\left[ \nabla_x \psi^{0} (t^\eps,x^{\eps})\right]}  } .
\end{equation*}
As the sequence $v^\eps$ is \textit{bounded} by (H0), passing to the limit $ \eps \to 0 $ thanks to the local uniform convergence yields 
\begin{equation*}
 \partial_t \psi^{0} (t^0,x^0)+ \mathcal{H}\left( \nabla_x \psi^{0} (t^0,x^0) \right)  + r \geq 0.
\end{equation*}

\end{proof}

{\bf \# Step 2: Viscosity Subsolution.} 

\qquad

Here comes a slight distinction between the cases $r > 0$ and $r =0$. Indeed, one gets less information (but enough) when the nonlinearity is present since the limit equation is an obstacle problem \eqref{varHJ}, similarly to \cite{Evans-Souganidis}.

\begin{lemma}\label{sub}
Suppose that $r>0$. On $ \left\lbrace \varphi^0 > 0 \right\rbrace \cap \left( \R^{+*} \times \R^n \right)$, the function $ \varphi^0 $ solves the following equation in the viscosity sense : 
\begin{equation*}
\forall (t,x) \in \left\lbrace \varphi^0 > 0 \right\rbrace \cap \left( \R^{+*} \times \R^n \right), \qquad \partial_t \varphi^0 + \mathcal{H}  \left( \nabla_x \varphi^0 \right) + r \leq 0.
\end{equation*}
In the case $r=0$, this same subsolution property holds in the full space $\R^{+*} \times \R^n$. 
\end{lemma}

\begin{proof}[\bf Proof of Lemma \ref{sub}]
Let $\psi^0 \in \mathcal{C}^2 \left( \R^{+*} \times \R^n \right)$ be a test function such that $ \varphi^0 - \psi^0 $ has a local maximum a $(t^0,x^0)$. We want to show that
\begin{equation*}
\qquad \partial_t \psi^0 (t^0,x^0) + \mathcal{H} \left( \nabla_x \psi^0 (t^0,x^0) \right) + r \leq 0.
\end{equation*}
Following the same steps as for Lemma \ref{super}, there exists $(t^{\eps},x^{\eps},v^{\eps}) \in \R^{+*} \times \R^n \times V$ with $(t^{\eps},x^{\eps}) \to (t^0,x^0)$ and a bounded sequence $v^{\eps}$ such that at the point $(t^{\eps},x^{\eps},v^{\eps})$:

\begin{equation}\label{subfinal}
\partial_t \psi^{\eps}  + \mathcal{H} \left( \nabla_x \psi^{\eps}  \right) + r  \leq \frac{\mathcal{P}\left(Q_{\left[ \nabla_x \psi^{\eps} (t^\eps,x^{\eps},v^{\eps}) \right]}  \right)(v^{\eps})}{Q_{\left[ \nabla_x \psi^{\eps} (t^\eps,x^{\eps},v^{\eps}) \right]} (v^\e) } - \frac{\mathcal{P}\left(Q_{\nabla_x \psi^{0} (t^\eps,x^{\eps}) } \right)(v^{\eps})}{Q_{\left[ \nabla_x \psi^{0} (t^\eps,x^{\eps})\right]}(v^\e)  } + r \rho^\eps(t^{\eps},x^{\eps})
\end{equation}

If $r=0$, one obtains directly with the uniform convergences that $\varphi^0$ is a subsolution of $\partial_t u  + \mathcal{H} \left( \nabla_x u  \right) = 0$ in $(t_0,x_0)$, as for Lemma \ref{super}.

We now come to the case $r >0$. Suppose now that $\varphi^{0}(t^0,x^0) > 0$, we have by the uniform convergence of $\varphi^\eps$ (up to extraction) that for sufficiently small $\e$, $\forall v \in V, \varphi^\eps( t^\e , x^\e ,v ) > 0$. The Lebesgue dominated convergence theorem gives
\begin{equation*}
\lim_{\eps \to 0} \rho^{\eps}(t^\eps , x^{\eps}) = \lim_{\eps \to 0} \int_{V} M(v) e^{ - \frac{\varphi^{\eps} (t^{\eps},x^{\eps},v) }{\eps} } dv = 0.
\end{equation*}
As a consequence, passing to the limit $\eps \to 0$ in \eqref{subfinal} yields 
\begin{equation*}
 \partial_t \psi^{0} (t^0,x^0)+ \mathcal{H} \left( \nabla_x \psi^{0} (t^0,x^0) \right)  + r \leq 0,
\end{equation*}
and the Lemma \ref{sub} is proved.
\end{proof}

\subsection{Uniqueness of the viscosity solution.}

Before referring to an uniqueness property for \eqref{standHJ} and \eqref{varHJ}, we have to check the initial conditions in the viscosity sense. We perform the proof in the variational case $(r > 0)$, the other one is similar. One has to check, in the viscosity sense
\begin{equation}\label{inicond1}
\min\left( \min\left \lbrace \partial_t \varphi^0 + \mathcal{H} \left(\nabla_x \varphi^0 \right) + r , \varphi^0 \right\rbrace , \varphi^0 - \varphi_0 \right) \leq 0, \qquad \textrm{in } \lbrace t = 0 \rbrace \times \R^n,
\end{equation}
and
\begin{equation}\label{inicond2}
\max\left( \min\left \lbrace \partial_t \varphi^0 + \mathcal{H} \left(\nabla_x \varphi^0 \right) + r , \varphi^0 \right\rbrace , \varphi^0 - \varphi_0 \right) \geq 0, \qquad \textrm{in } \lbrace t = 0 \rbrace \times \R^n.
\end{equation}
Since \eqref{inicond2} can be derived on the same model, we compute \eqref{inicond1} only. Let $\psi^0 \in \mathcal{C}^2\left( \R^+ \times \R \right)$ be a test function such that $\varphi^0 - \psi^0$ has a strict local maximum in $(0,x_0)$. We have to prove that either
\begin{equation*}
\varphi^0(0,x_0) - \varphi_0 (x_0)\leq 0,
\end{equation*}
or if $\varphi^0(0,x_0) > 0$, then
\begin{equation*}
\partial_t \psi^0(0,x_0) + \mathcal{H}  \left( \nabla_x \psi^0 (0,x_0)\right) + r \leq 0.
\end{equation*}
Suppose then that 
\begin{equation}\label{hyp}
\varphi^0(0,x_0) > \max \left( \varphi_0 (x_0) , 0 \right).
\end{equation}
Using the same arguments as in {\bf \# Step 2} above, we have a sequence $(t_\e,x_\e)$ which tends to $(0,x_0)$ as $\e\to 0$ and a converging sequence $v_\e$ such that $(t_\e,x_\e,v_\e)$ maximizes $\varphi^\e - \psi^\e$. The key point to be noticed is that there exists a sequence $\eps_n \to 0$ and a subsequence $(t_{\e_n},x_{\e_n},v_{\e_n})$ of $(t_\e,x_\e,v_\e)$ such that $t_{\e_n} > 0$.

Indeed, suppose $t_\e=0$ when $\eps$ is sufficiently small. Then for all $(t,x,v)$ in some neighborhood of $(0,x_\e,v_\e)$, one has 
\begin{equation*}
\varphi^\eps \left( 0 , x_\e , v_\e \right) - \psi^\eps \left( 0 , x_\e , v_\e \right) \geq \varphi^\eps \left( t , x , v \right) - \psi^\e \left( t , x , v \right). 
\end{equation*}
Passing to the limit $\eps \to 0$ thanks to the local uniform convergence and setting $(t,x) = (0,x_0)$, we get 
\begin{equation*}
\varphi_0(x_0) - \psi^0 \left( 0 , x_0 \right) \geq \varphi^0(0,x_0) - \psi^0(0,x_0), 
\end{equation*}
and this contradicts \eqref{hyp}. The conclusion is then similar as in {\bf \# Step 2} above since along $(t_{\e_n},x_{\e_n},v_{\e_n})$, Equation \eqref{subfinal} holds.


From Section \ref{EvPb}, the Hamiltonian $\mathcal{H}$ is a Lipschitz function of $p$. As a consequence, we know from \cite{Evans-book,Evans-Souganidis} that there exists a unique viscosity solution of \eqref{standHJ} and \eqref{varHJ}. It yields that all the sequence $\varphi^\eps$ converges locally uniformly to $\varphi^0$.
 
\section{The eigenvalue problem (H4).}\label{EvPb}

In this Section, we discuss the spectral problem of Hypothesis (H4). Existence basically relies on compactness, positivity, and the Krein-Rutman theory. As a complement, we also provide some qualitative properties of the resulting Hamiltonian. In the next Proposition, we treat the case when $P$ is compact and strongly positive. This is natural for kernel operators.

\begin{proposition}\label{Pcompact}
Let $V$ be a bounded velocity domain. Suppose that $P : \mathcal{C}^0(V) \mapsto \mathcal{C}^0(V)$ is a linear, compact, and strongly positive operator.
Moreover, if $r=0$ we require that there exists a constant $c$ such that $P(f) \geq c M(v) \int_V f dv$. Then the spectral problem of Hypothesis (H4) has a solution.
\end{proposition}

\begin{proof}[\bf Proof of Proposition \ref{Pcompact}]

Let us first recall and define
\begin{equation*}
\mathcal{P}(f) = P(f) + r M(v) \int_V f(v) dv, \qquad  \overline \Sigma = \Sigma + r. 
\end{equation*}
Note that since $V$ is bounded, $\mathcal{P}$ is also a compact operator. For all $p\in \R^n$, we are seeking $\mathcal{H}(p)$ such that there exists a positive function $Q \in \mathcal{C}^0(V)$ such that: 
\begin{equation}\label{rel1}
\forall v \in V, \qquad \mathcal{P}(Q)(v) = \left(  \bar \Sigma(v) + \mathcal{H}(p) - v \cdot p \right) Q(v).
\end{equation}
As in similar problems \cite{Perthame, Henkel}, we will use the Krein-Rutman Theorem \cite{Krein}. To make it appear, we denote $\mathcal{A}_\lambda(v) := \bar \Sigma(v) +  \lambda  - v \cdot p$. Note that since $V$ is bounded, one can guaranty the positivity of $\mathcal{A}_\lambda$ for all $\lambda > \lambda^* := \sup_{v \in V} \left( v \cdot p - \Sigma(v) \right)$. We now consider the following operator $T$:
\begin{equation*}
\forall \Phi \in \mathcal{C}^0(V), \qquad \forall v \in V, \qquad T( \Phi ) (v) = \frac{\mathcal{P} \left( \Phi \right) (v)}{\mathcal{A}_\lambda(v)}. 
\end{equation*}
Then, the relation \eqref{rel1} writes :
\begin{equation}\label{eq:T}
\forall v \in V, \qquad T( \Phi )(v) = \Phi (v).
\end{equation}
To solve this eigenvalue problem, we are now ready to apply the Krein-Rutman Theorem \cite{Krein}. Indeed, $T$ is also a strongly positive compact operator. We work on the total cone of positive continuous functions $K =  \mathcal{C}_{+}^{0} \left( V \right)$ to find $\Phi_\lambda \in K$ which solves:
\begin{equation*}
\forall v \in V, \qquad T(\Phi_\lambda)(v) = \mu_\lambda \Phi_\lambda(v),
\end{equation*}
where $\mu_\lambda$ is thus the principal eigenvalue of the operator $T$. We assume w.l.o.g. that $\int_{V} \Phi_{\lambda}(v') dv' = 1$. 

We can do the same for the adjoint operator of $T$, which is given by 
\begin{equation*}
\forall \Psi \in \mathcal{C}^0(V), \qquad T^{*}( \Psi ) = \mathcal{P}^{*} \left( \frac{\Psi}{\mathcal{A}_{\lambda}} \right).
\end{equation*}
From the same reasons as before for the direct problem, we can solve this latter eigenvalue problem to have both 
\begin{equation*}
T( \Phi_{\lambda} ) = \mu_\lambda \Phi_{\lambda}, \qquad T^{*}( \Psi_{\lambda} )= \mu_\lambda \Psi_{\lambda},
\end{equation*}
and the normalization $\left\langle \Psi_{\lambda} \vert \Phi_{\lambda} \right\rangle = \int_V \Psi_{\lambda} (v)  \Phi_{\lambda} (v) \; dv = 1$.

We will now prove that for all $p \in \R^n$, there exists only one $\lambda := \mathcal{H}(p)$ such that $\mu_\lambda = 1 $. For this purpose, we study the function $\mu : \lambda \mapsto \mu_\lambda$ on the set $\left] \lambda^* , + \infty \right[$.

First, let us prove that $\mu$ is decreasing.
To prove this point, we use the adjoint eigenvalue problem, see \cite[Chapter 4]{Perthame} for another example of utilization in the study of size-structured models via the relative entropy method. 
Differentiating the first one with respect to $\lambda$, and taking the duality product with $\Psi_{\lambda}$ on the left, we obtain
\begin{equation*}
\left\langle \Psi_{\lambda} \Big\vert \frac{d T }{d \lambda} ( \Phi_{\lambda} ) \right\rangle + \left\langle \Psi_{\lambda}\Big\vert T\left( \frac{d \Phi_{\lambda} }{d \lambda} \right) \right\rangle = \frac{d \mu_\lambda}{d \lambda}  \left\langle \Psi_{\lambda} \Big\vert \Phi_{\lambda}  \right\rangle + \mu_\lambda \left\langle \Psi_{\lambda} \Big\vert \frac{d \Phi_{\lambda} }{d \lambda}  \right\rangle,
\end{equation*}
from what we deduce, using $\left\langle \Psi_{\lambda} \Big\vert T\left( \frac{d \Phi_{\lambda} }{d \lambda} \right) \right\rangle =  \left\langle T^{*} \left( \Psi_{\lambda} \right) \Big\vert \frac{d \Phi_{\lambda} }{d \lambda} \right\rangle = \mu_\lambda \left\langle \Psi_{\lambda} \Big\vert \frac{d \Phi_{\lambda}}{d \lambda}  \right\rangle $ and recalling the normalization of $\Psi_\lambda$,  
\begin{equation*}
\frac{d \mu_\lambda}{d \lambda} = \left\langle \Psi_{\lambda} \Big\vert \frac{d T}{d \lambda} \left( \Phi_{\lambda} \right) \right\rangle.
\end{equation*}
As a consequence, as 
\begin{equation*}
\forall \Phi \in \mathcal{C}^0(V), \qquad \frac{d T}{d \lambda} ( \Phi )= -  \mathcal{P} \left( \frac{\Phi}{\left( \mathcal{A}_\lambda \right)^2} \right) 
\end{equation*}
is a negative operator, we deduce that $\mu$ is decreasing. 

We now focus on the limits of $\mu$ towards the boundary of $\left] \lambda^* , + \infty \right[$. From equation \eqref{eq:T}, we deduce
\begin{equation*}
\int_{V} \frac{ \mathcal{P}( \Phi_{\lambda})(v') }{\mathcal{A_\lambda}(v')} dv' = \mu_\lambda.
\end{equation*}
We have $\left\Vert \left(\mathcal{A}(v) \right)^{-1} \right\Vert_\infty \underset{\lambda \to \infty}{\longrightarrow} 0$, so that necessarily $\lim_{\lambda \to + \infty} \mu_\lambda = 0$.

Using Fatou's lemma, we get, with $\omega = r$ if $r >0$, $\omega = c$ else: 
\begin{multline*}
+ \infty = \int_{V} \text{lim inf}_{\lambda \to \lambda^*} \left( \frac{\omega M(v')}{\mathcal{A}_\lambda(v')} \right) {dv'} \leq \displaystyle \int_{V} \text{lim inf}_{\lambda \to \lambda^*} \left( \frac{ \mathcal{P}\left( \Phi_{\lambda} \right)(v') }{\mathcal{A}_\lambda(v')} \right) {dv'} \\\leq \text{lim inf}_{\lambda \to \lambda^*} \left(  \int_{V}  \frac{ \mathcal{P}( \Phi_{\lambda} )(v') }{\mathcal{A}_\lambda(v')} {dv'} \right) = \displaystyle \text{lim inf}_{\lambda \to \lambda^*} \mu_\lambda.
\end{multline*}
Finally, we obtain the existence and uniqueness of $\mathcal{H}(p)$ for all $p \in \R^n$. One associated eigenvector is given by $Q_p = \Phi_{\mathcal{H}(p)}$.

\end{proof}

\begin{remark}
With a supplementary regularization argument, the proof can be adapted replacing $\mathcal{C}^0(V)$ by $L^1(V)$. The assumption concerning the existence of a coercivity constant $c$ when $r=0$ may be relaxed in some particular cases. These technical points are not our purpose here, so we do not address these issues further.
\end{remark}
\begin{example}

Proposition \ref{Pcompact} (and its extension to $L^1(V)$) solves the case of kernel integral operators if one assume some supplementary hypothesis on the positive kernel $K$ which ensures the compactness of the operator $\mathcal{P}$. As an example 
assuming 
\begin{equation*}
\int_{V} \sup_{v' \in V} \left( K(v,v') \right) dv < + \infty, 
\end{equation*}
we ensure the compactness of $\mathcal{P}$, see \cite{Degond}. 


In the particular case where $L$ is a BGK operator given by $L(f) := M(v) \left( \int_V f(v') dv' \right) - f$, the kernel of $\mathcal{L}$ is $K(v,v') := (1+r) M(v)$. The compactness holds. Using the scaling property of Proposition \ref{propertiesH} below with $V = [-1;1]$ and $n=1$, and the Hamiltonian derived in the one-dimensional case \cite{Bouin-Calvez}, one could find 
\begin{equation*}
\forall p \in \R^n, \qquad \mathcal{H} (p) = \frac{p}{\tanh \left(\frac{p}{1+r}\right)} - (1+r),
\end{equation*}
We can also notice that in this case, the eigenfunctions are explicit up to the knowledge of the eigenvalue. We have $ \mu_\lambda = \int_{V} \frac{M(v)}{1 - \lambda - v \cdot p} dv $, so that $\mu_\lambda = 1$ gives the dispersion relation found in \cite{Bouin-Calvez}:
\begin{equation*}
\int_{V} \frac{M(v)}{1 - \lambda - v \cdot p} dv  = 1.
\end{equation*}
The associated eigenvectors are:
\begin{equation*}
Q_{p}(v) = \frac{M(v)}{1 + \mathcal{H}(p) - v \cdot p}, \qquad W_{p}(v) = \frac{1}{1 + \mathcal{H}(p) - v \cdot p} \cdot \left( \int_{V} \frac{M(v)}{(1 + \mathcal{H}(p) - v \cdot p)^2} \text{dv} \right)^{-1}
\end{equation*}
where the latter solves the adjoint problem.


\end{example}
We now prove a similar result in the case of an elliptic operator in a bounded domain.

\begin{proposition}\label{PMcompact}
Let $V$ be a bounded smooth domain and $D(v)$ is a uniformly positive definite diffusivity matrix. Suppose $P(f) := \nabla_v \left( D(v) \nabla_v f \right)$, with Neumann boundary conditions on $\partial V$. Then the eigenvalue problem \eqref{eigenpb} has a solution.
\end{proposition}

\begin{proof}[\bf Proof of Proposition \ref{PMcompact}]
The eigenvalue problem to be solved can be written 
\begin{equation*}
 -  \nabla_v \cdot \left( D(v) \nabla_v Q \right) - r \left( M(v) \int_V Q(v') dv' - Q \right)+ \left( \mathcal{H}(p) - v \cdot p \right) Q = 0.
\end{equation*}
Suppose first that $r=0$. In this case, the Krein-Rutman theorem \cite{Krein} on the cone $K =  \mathcal{C}_{+}^{0} \left( V \right)$ gives the result. Indeed, take a sufficiently large $\mathcal{H}(p)$ such that the operator has an inverse. By the strong maximum principle and Neumann boundary conditions the resolvant is then compact and positive.

\noindent Suppose now that $r >0$. One can assume that $\int_V Q(v') dv' = 1$. One then has to solve the following nonhomogeneous problem
\begin{equation}
 -  \nabla_v \cdot \left( D(v) \nabla_v Q \right) + \left( r + \mathcal{H}(p) - v \cdot p\right) Q = r M(v).
\end{equation}
But, by the strong maximum principle and the Neumann boundary conditions, and since $M$ is nonnegative, we deduce that for sufficiently large $\mathcal{H}(p)$, there exists a unique positive solution $Q_p$ to the latter equation. We now have to solve, as for Proposition \ref{Pcompact}, the dispersion relation $\int_V Q(v) dv = 1$ to prove that there is only one $\mathcal{H}(p)$ such that the relation holds. For this purpose, similarly to the proof of Proposition \ref{Pcompact}, we define $Q_\lambda$ solving
\begin{equation}\label{eq:eigenlapl}
 -  \nabla_v \cdot \left( D(v) \nabla_v Q_\lambda \right) + \left( r + \lambda - v \cdot p\right) Q_\lambda = r M(v).
\end{equation}
for some parameter $\lambda$ sufficiently large. Differentiating \eqref{eq:eigenlapl} with respect to $\lambda$, one finds
\begin{equation}
 -  \nabla_v \cdot \left( D(v) \nabla_v \frac{d Q_\lambda}{d \lambda} \right) + \left( r + \lambda - v \cdot p\right) \frac{d Q_\lambda}{d \lambda} = - Q_\lambda.
\end{equation}
As a consequence, $\frac{d Q_\lambda}{d \lambda} < 0$, and thus the application $\lambda \mapsto \int_V Q_\lambda dv$ is decreasing. Now integrating \eqref{eq:eigenlapl} with respect to $v$, we deduce that 
\begin{equation}
 \int_V Q_\lambda(v) dv \leq \frac{r}{\lambda + r - V_{\text{max}} \vert p \vert} \longrightarrow 0, 
\end{equation}
 as $\lambda$ goes to $+ \infty$. Dividing \eqref{eq:eigenlapl} by $r + \lambda - v \cdot p$, and integrating over $V$, we find
\begin{equation}
 \int_V Q_\lambda (v) dv = \int_V \frac{r M(v)}{r + \lambda - v \cdot p} dv + \int_V \left[ \frac{p \cdot \nabla D}{\left( r + \lambda - v \cdot p \right)^2} + \frac{ 2 \vert p \vert^2 D(v)}{\left( r + \lambda - v \cdot p \right)^3} \right] Q_\lambda dv,
\end{equation}
so that as $\lambda$ tends to $V_{\text{max}} \vert p \vert - r$ by larger values, $  \int_V Q_\lambda (v) dv $ tends to $ +\infty$ (since the last integral of the r.h.s is positive for sufficiently small values of $\lambda$).
By a monotonicity argument, we are able to conclude that for all $p \in \R^n$, the dispersion relation $\int_V Q_\lambda(v) dv = 1$ has only one solution, that is called $\mathcal{H}(p)$.
\end{proof}

 \begin{example}In the simple case given by $P(f) = \alpha \Delta f$, the solution of the eigenvalue problem \eqref{eigenpb}
can be written down with Airy functions. It appears in some reaction-diffusion-mutation models without maximum principle.

\end{example}

We finish this section investigating some relevant properties of the Hamiltonian $\mathcal{H}$.

\begin{proposition}\label{propertiesH}
Assume that (H4) holds. Then the Hamiltonian $\mathcal{H}$ is a Lipschitz continuous. It satisfies $\mathcal{H}(0) = 0$ and $\nabla_p \mathcal{H}(0) = 0$. Finally, it also satisfies the scaling property 
\begin{equation*}
\forall \mu \in \R^*, \qquad  \mathcal{H}_{\mu L} = \mu \mathcal{H}_L \left( \frac{\cdot}{\mu} \right),
\end{equation*}
where we denote by $\mathcal{H}_{\mu L}$ the Hamiltonian associated to some operator $L$.
\end{proposition}

\begin{proof}[\bf Proof of Proposition \ref{propertiesH}]

We get that $\mathcal{H}(0) = 0$ as a byproduct of the integration of \eqref{eigenpb} over $V$:
\begin{equation*}
\forall p \in \R^n, \qquad \vert \mathcal{H}(p) \vert = \left\vert \left( \int_{V} v Q_p(v) dv \right) \cdot p \right\vert \leq V_{max} \vert p \vert.
\end{equation*}
This latter inequality prove the sublinear behavior of the Hamiltonian. To prove the Lipschitz character of the Hamiltonian, we again use the adjoint formulation of  \eqref{eigenpb}. Indeed, we can solve it as for the direct problem, so that there exists $W_p$ such that 
\begin{equation*}
\mathcal{P}\left( Q_{p} \right)(v) = \left(\Sigma(v) + \mathcal{H}(p) - v \cdot p \right) Q_{p}(v), \qquad \mathcal{P}^{*}( W_{p} )(v) = \left(\Sigma(v) + \mathcal{H}(p) - v \cdot p\right) W_{p}(v). 
\end{equation*}
Differentiating these two equalities with respect to $p$, we get
\begin{equation*}
\left(\Sigma(v) + \mathcal{H}(p) - v \cdot p \right)  \frac{d Q_{p}}{d p} + Q_p \left( \nabla_p \mathcal{H} - v \right) = \mathcal{P} \left( \frac{d Q_{p}}{d p} \right),
\end{equation*}
As previously performed, we integrate against $W_p$,
\begin{equation*}
\left\langle  (\Sigma(v) + \mathcal{H}(p) - v \cdot p) W_p \Big\vert \frac{d Q_{p}}{d p} \right\rangle + \langle W_{p} \vert Q_{p} \left( \nabla_p \mathcal{H} - v \right) \rangle = \left\langle W_{p} \Big\vert \mathcal{P} \left( \frac{d Q_{p}}{d p} \right) \right\rangle
\end{equation*}
so that
\begin{equation}\label{ordre1}
\langle W_{p} \vert Q_{p} \left( \nabla_p \mathcal{H} - v \right) \rangle = 0 \Longleftrightarrow \nabla_p \mathcal{H} = \frac{ \langle W_{p} \vert v Q_{p} \rangle }{ \langle W_{p} \vert Q_{p} \rangle } \Longleftrightarrow \vert \nabla_p \mathcal{H} \vert \leq V_{max},
\end{equation}
and this gives that $\mathcal{H}$ is Lipschitz. Moreover, we always have $Q_0(v) = M(v)$ and $W_0 = 1$, the last one coming from the conservation property. Thus, 
\begin{equation*}
\nabla_p \mathcal{H} ( 0 )=\langle W_{0} \vert v Q_{0} \rangle = \int_{V} v M(v) dv = 0.
\end{equation*}

The last point follows from the uniqueness of the solution $\mathcal{H}$ of the eigenvalue problem \eqref{eigenpb}. Indeed, we have for all $\mu \in \R^*$, 
\begin{equation*}
\forall v \in V, \qquad \mu \mathcal{H}\left(\frac{p}{\mu}\right) = v \cdot p + \frac{\mu \mathcal{L}(\hat Q_p)}{\hat Q_p}(v), 
\end{equation*}
with $\hat Q_p = Q_{\mu p}$, where $Q_p$ is an eigenvector for $\mathcal{H}(p)$. 
\end{proof}

\begin{remark}
\begin{enumerate}
\item Here appears one of the most striking conclusion of our study. The Hamiltonian of the limiting equation in the large deviation regime is Lipschitz continuous. It differs strongly from the case of th Fisher-KPP equation which is obtained as the dryft-diffusion limit of \eqref{KinEq}. This means that the diffusion limit is not compatible with large deviations and thus propagation of fronts.
\item A classical attempt in the Hamilton-Jacobi theory is the convexity of the Hamiltonian. In \cite{Bouin-Calvez}, the authors manage to prove that for the simplest BGK case, it is indeed convex. However, it seems not to be an easy issue in general. We were not able to conclude if the Hamiltonian is convex or not. 

\item Thanks to the Proposition \ref{propertiesH}, we can replace $\mathcal{L}$ by a more "barycentric" one $\mathcal{L}_r$:
\begin{equation}\label{barop}
\mathcal{L}_r (f) = \frac{L(f) + r \left( M(v) \rho - f \right)}{1+r},
\end{equation}
solve the underlying eigenvalue problem \eqref{eigenpb} to get an Hamiltonian $\mathcal{H}_{\mathcal{L}_r}$ and deduce the following relation
\begin{equation*}
\forall p \in \R^n, \qquad \mathcal{H}( p ) = (1+r) \mathcal{H}_{\mathcal{L}_r}\left( \frac{p}{1+r} \right).
\end{equation*}
The latter identity can be useful for example when $L$ is also a BGK operator, as in \cite{Bouin-Calvez}.
\item One could want to derive a expression of the total Hamiltonian which only depends on the Hamiltonian associated to $L$. However, even though the BGK operator and $L$ commute, we cannot generally derive an expression for the Hamiltonian of their sum. Indeed, the construction of solutions of the spectral problem shows that the Hamiltonians appear as spectral radius of operator. Basically, it is not possible to obtain a exact general formula for the spectral radius of the sum of two operators.

\end{enumerate}
\end{remark}

\section{Asymptotics, numerics and comments.}\label{Asymptotics}
\subsection{Further asymptotics.}

This subsection aims at proving some convergence results for the total density $\rho^\eps$ in both regions $\lbrace \varphi^0 = 0 \rbrace$ and $\lbrace \varphi^0 > 0 \rbrace$.

\begin{proposition}\label{zones1}
Let $\varphi^\eps$ be the solution of \eqref{KinEqPhi}. Theorem \ref{HJlimit} says that it converges locally uniformly towards a nonpositive $\varphi^0$, the unique viscosity solution of \eqref{varHJ}. Uniformly on compact subsets of $\text{Int}\left\lbrace \varphi^0 > 0\right\rbrace$, the convergence $\lim_{\eps \to 0} f^{\eps} = 0$ holds, and is exponentially fast.
\end{proposition}

\begin{proposition}\label{zones2}
Let $\varphi^\eps$ be the solution of \eqref{KinEqPhi}. Assume now that $r > 0$. Then, uniformly on compact subsets of $\text{Int} \left\lbrace \varphi^0 = 0\right\rbrace $, 
\begin{equation*}
\lim_{\eps \to 0} \rho^\eps = 1, \qquad \lim_{\eps \to 0} f^\eps \left( \cdot , v \right) = M(v).
\end{equation*}

\end{proposition}

\begin{remark}

Assume that $r=0$ and $\mathcal{H}$ is convex. Then the mass stays at its initial position:
\begin{equation*}
\lbrace (t,x) \in \R^+ \times \R^n \, \vert \, \varphi^0(t,x) = 0 \rbrace = \R^+ \times \lbrace x \in \R^n \, \vert \, \varphi_0(x) = 0 \rbrace.
\end{equation*}
Indeed, one can write the solution of the standard Hamilton-Jacobi equation \eqref{standHJ} with the Hopf-Lax formula:
\begin{equation*}
\varphi^0(t,x) = \inf_{\gamma \in X} \left\lbrace \varphi_0 \left( \gamma(0) \right) + \int_{0}^{t} \mathcal{M} \left( \dot \gamma (t) \right) dt  \, \Big\vert \, \gamma(t) = x \right\rbrace,
\end{equation*}
where $\mathcal{M}$ is the Lagrangian associated to $\mathcal{H}$. Since $\nabla_p \mathcal{H} (0) = 0$ and $ \mathcal{H} $ is strictly convex, so does $\mathcal{M}$, and as a consequence $\mathcal{M}$ is positive away from $0$. We deduce that 
\begin{equation*}
\varphi^0(t,x) = 0 \quad \Longleftrightarrow \quad t \in \R^+ \quad \textrm{and} \quad \varphi_0 (x) = 0.
\end{equation*}

%
\end{remark}

\begin{proof}[\bf Proof of Proposition \ref{zones1}]
Let $K$ be a compact subset of $\text{Int}\left\lbrace \varphi^0 > 0\right\rbrace$. The local uniform convergence of $\varphi^\eps$ towards $\varphi^0$ ensures that there exists $\delta > 0$ such that for sufficiently small $\eps > 0$, $\varphi^\eps \geq \delta$ on $K$. As a consequence,
\begin{equation*}
\forall (t,x,v) \in K \times V, \qquad f^\eps(t,x,v) = M(v) \exp\left( - \frac{\varphi^\eps(t,x,v)}{\eps} \right)  < M(v)\exp\left( - \frac{\delta}{\eps} \right)  \underset{\eps \to 0}{\longrightarrow} 0.
\end{equation*}
\end{proof}

\begin{proof}[\bf Proof of Proposition \ref{zones2}]
We develop similar arguments as in \cite{Evans-Souganidis}. Note that it suffices to prove the result when $K$ is a cylinder. Let $(t_0,x_0) \in \text{Int}\left( K \right)$ and the test function 
\begin{equation*}
\forall (t,x) \in K, \qquad \psi^0(t,x) = \vert x - x_0 \vert^2 + \left( t - t_0 \right)^2.
\end{equation*}
We can define the same corrected test function $\psi^\eps$ as in the viscosity procedure of Section \ref{HJProof}. Since $\varphi^0 = 0$ on $K$, the function $\varphi^0 -  \psi^0$ admits a strict maximum in $(t_0,x_0)$. The locally uniform convergence of $\varphi^\eps -  \psi^0$ gives a sequence $(t^{\eps},x^{\eps},v^{\eps})$ of maximum points with $(t^{\eps},x^{\eps}) \to (t^0,x^0)$ and a bounded sequence $v^{\eps}$ such that at the point $(t^{\eps},x^{\eps},v^{\eps})$ one has (see \eqref{subfinal}):
\begin{equation*}
\partial_t \psi^{\eps}  + \mathcal{H} \left( \nabla_x \psi^{\eps}  \right) + r  \leq \frac{\mathcal{P}\left(Q_{\left[ \nabla_x \psi^{\eps} (t^\eps,x^{\eps},v^{\eps}) \right]}  \right)}{Q_{\left[ \nabla_x \psi^{\eps} (t^\eps,x^{\eps},v^{\eps}) \right]}  } (v^{\eps}) - \frac{\mathcal{P}\left(Q_{\nabla_x \psi^{0} (t^\eps,x^{\eps}) } \right)}{Q_{\left[ \nabla_x \psi^{0} (t^\eps,x^{\eps})\right]}  }(v^{\eps}) + r \rho^\eps(t^{\eps},x^{\eps}).
\end{equation*}
As a consequence, one has, since $r>0$, 
\begin{equation}\label{convepsrho}
\rho^\eps(t^{\eps},x^{\eps}) \geq 1 + o(1), \qquad \textrm{as } {\eps \to 0},
\end{equation}
and then $\lim_{\eps \to 0}  \rho^\eps(t^{\eps},x^{\eps}) = 1$ if one recalls $\rho^\eps \leq 1$ (which, again, is a consequence of the maximum principle). 

However, we need an extra argument to get the result in $(t_0,x_0)$ and then on all $K$. One has, for all $(t,x) \in K$,
\begin{equation*}
\partial_t \rho^\eps = - \frac{1}{\eps}   \int_V M(v) e^{- \frac{\varphi^\eps(t,x,v)}{\eps} } \partial_t  \varphi^\eps(t,x,v) dv. 
\end{equation*} 
As a consequence, for all $(t,x) \in K$,
\begin{equation*}
\left \vert \partial_t \rho^\eps (t,x) \right\vert \leq \frac{\Vert \partial_t \varphi^\eps (t,x,\cdot)\Vert_{L^\infty \left( V \right)}}{\eps} \int_V M(v) e^{- \frac{\varphi^\eps(t,x,v)}{\eps} } dv = \frac{\Vert \partial_t \varphi^\eps (t,x,\cdot)\Vert_{L^\infty \left( V \right)}}{\eps} \rho^\eps.
\end{equation*}
By the Grönwall Lemma, for all $t > t_0$ such that $(t,x^\eps) \in \text{Int}(K)$, 
\begin{equation*}
\rho^\eps(t^\eps , x^\eps) \leq \rho^{\eps}(t, x^\eps) e^{ \frac{\Vert \partial_t \varphi^\eps \Vert_{\infty,K}}{\eps}  \left( t_\eps - t\right)} \leq e^{ \frac{\Vert \partial_t \varphi^\eps \Vert_{\infty,K}}{\eps}\left( t_\eps - t\right)}.
\end{equation*}
This latter equation coupled to \eqref{convepsrho} gives that necessarily $\Vert \partial_t \varphi^\eps \Vert_{L^\infty \left( K \times V \right)} = \mathcal{O}(\eps)$. The same argument on the space derivatives gives the same conclusion for $\Vert \nabla_x \varphi^\eps \Vert_{L^\infty \left( K \times V \right)}$. Now, multiplying \eqref{KinEqPhi1} by $M(v) e^{ - \frac{\varphi^{\eps}}{\eps} }$ and integrating with respect to $v$ gives
\begin{equation*}
\forall (t,x) \in K, \qquad r \rho^\eps \left( 1 - \rho^\eps \right) \leq \left( \Vert \partial_t \varphi^\eps \Vert_{L^\infty} +  \vert v_{\max} \vert \cdot \Vert \nabla_x \varphi^\eps \Vert_{L^\infty} \right) \rho^\eps
\end{equation*}
from which we deduce that $\lim_{\eps \to 0} \rho^\eps(t,x) = 1$ locally uniformly on $K$ since $r>0$. This implies necessarily $\lim_{\eps \to 0} f^\eps(t,x,v) = M(v)$ locally uniformly on $K \times V$.

\end{proof}
\subsection{Study of the viscosity solution and of the speed of propagation.}

To be self-contained, we recall here how to study the propagation of the front after deriving the limit variational equation, in the case $r>0$. From Evans and Souganidis \cite{Evans-Souganidis}, we are able to identify the solution of the variational Hamilton-Jacobi equation \eqref{varHJ} using the Lagrangian duality. We recall the equation:
\begin{equation*}
\begin{cases}
\min\left \lbrace \partial_t \varphi^0 + \mathcal{H} \left(\nabla_x \varphi^0 \right) + r, \varphi^0 \right\rbrace = 0, \qquad  \forall (t,x) \in \R_+^* \times \R^n, \medskip\\
\varphi^{0}(0,x) = \varphi_0(x).
\end{cases}
\end{equation*}
We will suppose in this Subsection that the hamiltonian $\mathcal{H}$ is convex and is a function of $\vert p \vert$. The relevant result in the present context is the following 

\begin{proposition}[Speed of propagation]\label{nullset}
Assume that 
\begin{equation*}
\varphi_0(x) := \left\lbrace\begin{array}{lcl}
0& x=0\\
+ \infty & \text{ else}\\
\end{array}\right., 
\end{equation*}
and define $c^* = \inf_{p > 0}\left( \frac{\mathcal{H}(p) + r}{p} \right)$, see \cite{Fedotov,Fedotov99}. Then the nullset of $\varphi$ propagates at speed $c^*$ :
\begin{equation*}
\forall t \geq 0, \qquad \left\{ \varphi(t, \cdot ) = 0 \right\} = B(0, c^* t ).
\end{equation*}

\end{proposition}

\begin{proof}[\bf Proof of Proposition \ref{nullset}]

The Lagrangian associated to $\mathcal{H} + r $ is by definition 
\begin{equation*}
 \mathcal{L}(p):= \sup_{q \in \R^n} \left( p \cdot q - \mathcal{H}(q) - r  \right),
\end{equation*}
and one has, since $\mathcal{H}(q) = \mathcal{H}\left( \vert q \vert \right)$:
\begin{equation*}
\mathcal{L}(p) = \sup_{q \in \R^n} \left( \vert p \vert \vert q \vert \left( \frac{p}{\vert p \vert} \cdot \frac{q}{\vert q \vert} \right) - \mathcal{H}( \vert q \vert ) - r \right) = \sup_{q \in \R^n} \left( \vert p \vert \vert q \vert  - \mathcal{H}( \vert q \vert ) - r \right).
\end{equation*}
\begin{equation*}
\mathcal{L}(p) = 0 \quad \Longleftrightarrow \quad \vert p \vert =  \inf_{u > 0}\left( \frac{\mathcal{H}(u) + r}{u} \right) =  c^*.
\end{equation*}
To solve the variational Hamilton-Jacobi equation, let us define
\begin{equation*}
J(x,t) = \inf_{x \in X} \left \lbrace \int_{0}^{t}  \left[ \mathcal{L}( \dot x ) \right] ds \big\vert x(0) = x, x(t) = 0 \right \rbrace 
\end{equation*}
the minimizer of the action associated to the Lagrangian. Thanks to the so-called Freidlin condition, see \cite{Evans-Souganidis, Freidlin} we deduce that the solution of \eqref{varHJ} is 
\begin{equation*}
\varphi(x,t) = \max \left( J(x,t) , 0 \right) . 
\end{equation*}
The Lax formula gives 
\begin{equation*}
J(x,t) = \min_{y \in \R} \left\{ t \mathcal{L}\left( \frac{x-y}{t} \right) + \varphi_0(y)     \right\} = t \mathcal{L}\left( \frac{x}{t} \right)
\end{equation*}
thanks to the assumption on the initial condition. Finally, as $\mathcal{L}$ is increasing with $\vert p \vert$, the nullset of $\varphi$ is exactly $B(0, c^* t )$.

%
%
%

\end{proof}

\subsection{Numerical simulations}

We show in Figure \ref{num1} some numerical simulations of the evolution of the nullset of the solution of the variational Hamilton-Jacobi equation to illustrate our study. The speed of the front is easily numerically computable with this approach. When the Hamiltonian is not known explicitly, which is the most frequent case, it is still possible to solve numerically the spectral problem (H4) to obtain a numerical Hamiltonian, which can afterwards be used to compute the whole numerical solution. 
\begin{figure}[ht]\label{num1}
\begin{center}
\includegraphics[width= 0.49 \linewidth]{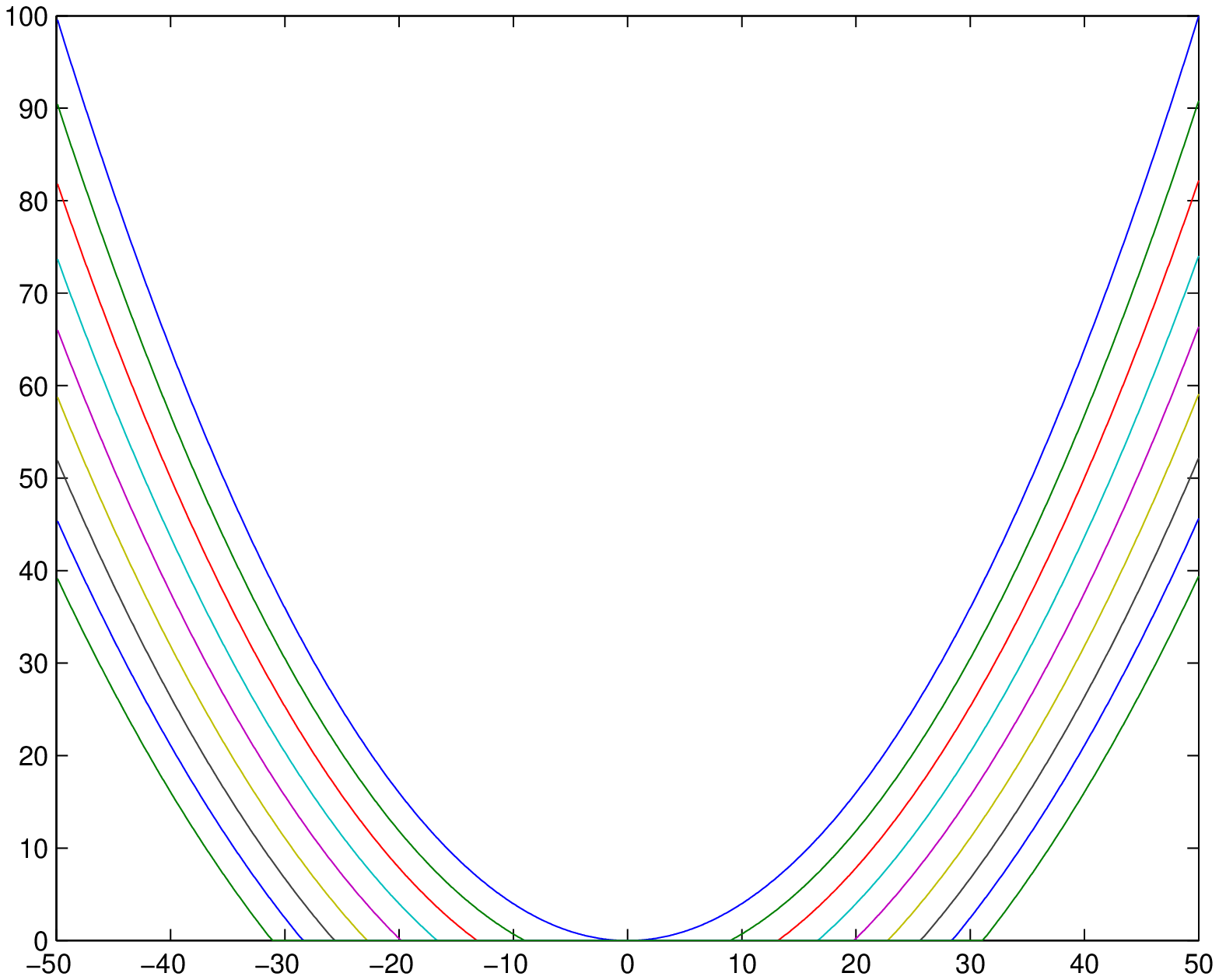}
\includegraphics[width= 0.49 \linewidth]{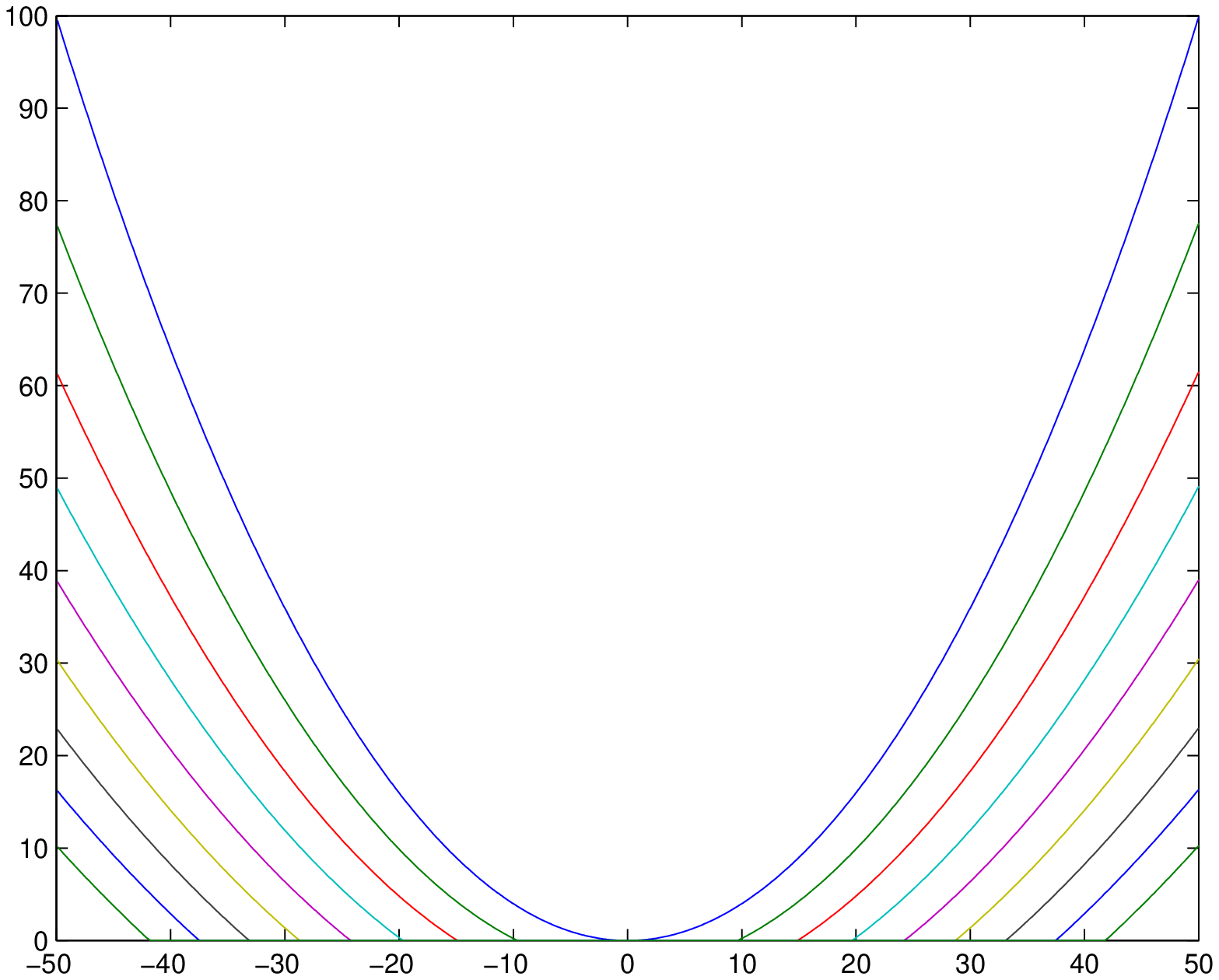}
\caption{Numerical simulations of the variational Hamilton-Jacobi equation with $r=1$, in the BGK case $\mathcal{H}(p) \equiv \frac{p - \tanh(p)}{\tanh(p)}$ (on the left) and in the "KPP case" $\mathcal{H}(p) = \vert p \vert^2$ (on the right). On both figures the linear propagation is noticed. In the "quadratic case" (KPP) the speed is larger than in the "at most linear" case.}
\end{center}
\end{figure}

\section{Remarks and perspectives in an unbounded velocity domain (\textit{e.g.} $V = \R^n$).}\label{Extensions}

In the previous Sections, the boundedness of the velocity space $V$ ( Hypothesis (H0) ) was a central hypothesis. Indeed, it gives immediately the compactness of operators to solve the spectral problem (H4), and facilitates the derivation of the uniform estimates of $\varphi^\eps$. Moreover, it automatically bounds the sequence $v^\eps$ in the viscosity procedure of Lemmas \ref{super} and \ref{sub}. This last property appears not to be true in general, see below.

In this last Section, we would like to comment on the case when $V$ is not bounded, and more precisely the case of the full space $V=\R^n$. We expect that, given that (H4) holds (which basically requires stronger assumptions on the operator $\mathcal{L}$ in the full space) the convergence result is still valid despite technicalities due to the unboundedness of the space. 

We first discuss the case of the transport-diffusion equation to illustrate the crucial character of (H4): The spectral problem (H4) does not have any non trivial solution in that case, and we show that the scaling $(t,x,v) \to \left( \frac{t}{\e} , \frac{x}{\e} , v \right)$ is not relevant. We then provide an example - the Vlasov equation - where the problem is compact in the velocity space. However, extending the convergence results in that case will need extra work and this issue will be discussed in a forthcoming work. Since we believe that this paper should be understood through examples, we end this Section with formal computations on a non-local convolution model.

\subsection{The Laplacian equation in an unbounded velocity domain.}

In this Subsection, we want to investigate the asymptotic properties of solutions of the following kinetic-diffusion equation
\begin{equation}\label{kindiff}
\forall (t,x,v) \in \R^+ \times \R^n \times \R^n, \qquad \partial_t f + v \cdot \nabla_x f = \sigma \Delta_v f.
\end{equation}
First of all, one can notice that the associated spectral problem of Hypothesis (H4), which writes
\begin{equation*}
\forall v \in \R^n, \qquad \sigma \Delta_v Q + ( v \cdot p ) \; Q = \mathcal{H}(p) Q
\end{equation*}
does not admit any nontrivial positive solutions. It relies on the lack of compactness of the Laplace operator on an unbounded domain. As a consequence, the method we have used before to average the velocity variable in the bounded velocity domain case cannot be applied here. We will now show that the scaling $(t,x,v) \to \left( \frac{t}{\e} , \frac{x}{\e} , v \right)$ is not well adapted and propose a more relevant scaling. In this case, as for the heat equation for example, one can guess this scaling by computing the fundamental solution of the kinetic diffusion operator. We recall this computation for the sake of completeness \cite{Kolmogorovalld}.

\begin{proposition}\label{propfond}

Let $f(t,x,v)$ be the solution of \eqref{kindiff} on $\R^+ \times \R^n \times \R^n$, associated to the initial data $\delta_x \delta_{v-w}$. Then 
\begin{equation*}
\forall (t,x,v) \in \R^{+*} \times \R^n \times \R^n, \qquad f_w(t,x,v) = \frac{\sqrt{3}}{ 2 \pi \sigma t^2 } \exp \left( - \frac{\left\vert v-w \right\vert^2 t^2 + 3 \vert 2x - (v+w)t \vert^2}{4 \sigma t^3} \right).
\end{equation*}
\end{proposition}

\begin{proof}[{\bf Proof of Proposition \ref{propfond}}]
This computation can be done using the Fourier transform $\mathcal{F}$ in space and velocity, since the operator is linear. One obtains
\begin{equation*}
\forall (t,k,p) \in \R^+ \times \R^n \times \R^n, \qquad \mathcal{F}(f)(t,k,p) = \exp\left(- i (p+kt) w\right) \exp \left( - \sigma t \left(  \left\vert p + \frac{kt}{2} \right\vert^2 + \vert k \vert^2 \frac{t^2}{12}\right) \right),
\end{equation*}
and the inverse Fourier transform can be easily computed using the invariances of Gaussians with respect to the Fourier transformations.

\end{proof}

We now perform an alternative scaling on this equation, namely $(t,x,v) \to \left( \frac{t}{\e} , \frac{x}{\eps^{2}} , \frac{v}{\eps} \right)$. Then the fundamental solution $f_0$ becomes  
\begin{equation*}
\forall (t,x,v) \in \R^{+*} \times \R^n \times \R^n, \qquad f_0^\e(t,x,v) = \frac{\sqrt{3}\e^2}{ 2 \pi \sigma t^2 } \exp \left( - \frac{1}{\e} \frac{\left\vert v \right\vert^2 t^2 + 3 \vert 2x - vt \vert^2}{4 \sigma t^3} \right).
\end{equation*}
In this framework, and only with this scaling, we recover the sharp front ansatz that we studied in the previous part of the article with a bounded domain, and the associated phase (which now depends on $v$) $\varphi^0(t,x,v) = \frac{\left\vert v \right\vert^2 t^2 + 3 \vert 2x - vt \vert^2}{4 \sigma t^3}$. It is also possible to obtain this result by performing the Hopf-Cole transformation in \eqref{kindiff} and then solving the limiting Hamilton-Jacobi equation on the phase $\varphi^0$ which reads:
\begin{equation*}
\forall (t,x,v) \in \R^+\times \R^n \times \R^n, \qquad \partial_t \varphi^0 + v \cdot \nabla_x \varphi^0 + \sigma \vert \nabla_v \varphi^0 \vert^2 = 0.
\end{equation*}
We obtained an example where the spectral problem has no solution (see also \cite{Bouin-Calvez} for another fundamental example), and this makes the information propagate as $x \sim t^2$: There is a front acceleration, as noticed for others models, see \cite{Bouin,Bouin-Calvez-Nadin,Bouin-Calvez-Grenier-Nadin}.

\subsection{The Vlasov-Fokker-Planck equation}

We would like now to comment on the Vlasov-Fokker-Planck equation, where the velocity operator provides enough compactness to solve the spectral problem (H4). 
Our equation reads 
\begin{equation}\label{Vlasov}
\forall (t,x,v) \in \R^+ \times \R^n \times \R^n, \qquad \partial_t f + v \cdot \nabla_x f = \nabla_v \cdot \left( \nabla_v f + \frac{v }{ \sigma^2} f \right). 
\end{equation}
The normalized stationary density is given by the Gaussian equilibrium $M(v) = \frac{1}{\sigma \sqrt{2 \pi}} \exp\left( - \frac{v^2}{2\sigma^2} \right)$. After performing the hyperbolic scaling and the kinetic WKB ansatz \eqref{WKBansatz}, it yields
\begin{equation*}
\forall (t,x,v) \in \R^+ \times \R^n \times \R^n, \qquad \partial_t \varphi^\eps + v \cdot \nabla_x \varphi^\eps = \frac{1}{\eps} \Delta \varphi^\eps - \frac{1}{\eps^2} \vert \nabla_v \varphi^\eps \vert^2 - \frac{1}{\eps} \nabla_v \psi \cdot \nabla_v \varphi^\eps, 
\end{equation*}
By parabolic regularity, one obtains, given an initial condition $\varphi_0(x,v) \in \mathcal{C}_b^2\left( \R^n \times \R^n \right)$, one unique bounded solution $\varphi^\eps$ in $\mathcal{C}^{2,\alpha} \left( \R^+ \times \R^n \times \R^n \right)$ for all $\eps > 0$. The spectral problem associated to \eqref{Vlasov} is :
\begin{equation}\label{SpecPb}
\nabla_v \cdot \left( \nabla_v Q_p + \frac{v }{ \sigma^2} Q_p \right) + \left( v \cdot p \right) Q_p = \mathcal{H}(p) Q_p,
\end{equation}
As a particular feature a the Gaussian case, one can solve \eqref{SpecPb} explicitely using the Fourier transformation. It yields the following eigenelements
\begin{equation*}
\mathcal{H}(p) = \sigma^4 \vert p \vert^2, \qquad Q_p(v) = \frac{1}{\sigma \sqrt{2 \pi}} \exp\left( - \frac{(v- \sigma^4 p)^2}{2\sigma^2} \right).
\end{equation*}
Hence, our hypothesis (H4) is fulfilled.

We shall comment here on the complications due to the unboundedness of the space. We cannot perform the same proof as for the proof of Theorem \ref{HJlimit}. Indeed, the sequence of approximated extremas in velocity, namely $v^\eps$ defined after \eqref{eqeta}, may not exist in general in a unbounded velocities setting. In particular, in this case, the correction $\eta$ is given by 
\begin{equation}
\eta(t,x,v) = - \ln \left( \frac{ Q_{\left[\nabla_x \psi^{0}(t,x)\right]}(v) }{M(v)} \right) = \sigma^2 v \cdot \nabla_x \psi^{0}(t,x) - \sigma^6 \vert \nabla_x \psi^{0}(t,x) \vert^2,
\end{equation}
which is linear in $v$, so that the function $\varphi^\eps - \eps \eta$ has no possible extrema in the velocity variable. This indicates that the correction term of order $\eps$ converges locally uniformly but not globally towards the corrector $\eta$. We postpone the analysis of this case to a forthcoming work.

\subsection{Formal computations on a confined non-local equation.}

We finish this paper with \textit{formal} computations on a case where the diffusive part of the operator is replaced by a nonlocal convolution operator. This is motivated by biological problems where mutations can have large range. We keep the drift part to ensure compactness. We are given a probability kernel $\mathrm{K}$ on $\R$, and we set 
\begin{equation*}
L(f) := \left( \mathrm{K} \star f - f \right) + \nabla \cdot \left( v f \right).
\end{equation*}
Solving the eigenvalue problem using the Fourier transform in the full space, we obtain that necessarily 
 \begin{equation*}
\mathcal{H}(p) = \mathrm{\hat K}(ip) - 1, \qquad F(Q_p)(\xi) =  \exp \left(  \int_{0}^{\xi} \frac{\mathrm{ \hat K}(\xi') -   \mathrm{ \hat  K}(ip)}{ \xi' - ip } d \xi' \right).
\end{equation*}
As $  \mathrm{ \hat  K}(ip) = \int_{V} \mathrm{ K}(x) e^{px} dx$, we observe that this would define an Hamiltonian on the zone where $K$ admits exponential moments. We have 
\begin{equation*}
Q_p(v) = \int_V  \exp \left(  \int_{0}^{\xi} \frac{ \mathrm{ \hat K}(\xi') -   \mathrm{ \hat K}(ip)}{ \xi' - ip } d \xi' \right) \exp \left(  i v \xi \right) d \xi,
\end{equation*}
where the last integral over the velocities has to be understood in the Fourier-Plancherel $L^2$ sense. One can easily prove that such a $Q_p$ is well normalized and real. The point which makes this Subsection be only formal is that we were not able to prove that such a $Q_p$ is indeed a positive eigenvector. Let us provide a few examples that strengthen this conjecture.

\begin{example} We now specify some convolution kernels.
\begin{enumerate}
\item $K(x) = \frac{1}{\sqrt{2 \pi}} e^{- \frac{x^2}{2} }$. Then $\mathcal{H}$ is well-defined on $\R$ and $H(p) = e^{ \frac{p^2}{2} } - 1.$
\item $K(x) = \frac12 e^{- \vert x \vert }$. Then $\mathcal{H}$ is well-defined on $]-1,1[$ and $\mathcal{H}(p) = \frac{p^2}{1 - p^2}.$
In this case, we can compute a bit further $F(Q_p)$ : 
\begin{equation*}
F(Q_p)(\xi) = \frac{1}{\left( 1 + \vert \xi \vert^2 \right)^{ \frac{1}{2(1-p^2)} } } \cdot \exp \left( i \frac{p}{p^2 - 1} \arctan(\xi) \right).
\end{equation*}
In particular, when $p=0$, one has 
\begin{equation*}
F(Q_p)(\xi) = \frac{1}{\left( 1 + \vert \xi \vert^2 \right)^{ \frac{1}{2} } },
\end{equation*}
which inverse Fourier transform can be computed with Airy functions and is positive. A numerical plot confirms formally the positivity of $Q_p$ (result not shown).

\end{enumerate}
\end{example}

\section*{Acknowledgement}

The author wishes to thank deeply Vincent Calvez to have suggested this problem to him and for stimulating and very interesting discussions about it.

\end{document}